\documentclass[12pt,a4paper]{amsart}
\usepackage{a4wide}
\usepackage{amsmath, amssymb, amsfonts,enumerate}
\usepackage[all]{xy}
\usepackage{amscd}
\usepackage{comment}
\usepackage{mathtools}
\usepackage{tikz} 
\usepackage{url}
\usepackage{hyperref}

\DeclarePairedDelimiter\floor{\lfloor}{\rfloor}

\newtheorem{theorem}{Theorem}[section]
\newtheorem{lemma}[theorem]{Lemma}
\newtheorem{corollary}[theorem]{Corollary}
\newtheorem{proposition}[theorem]{Proposition}

 \theoremstyle{definition}
 \newtheorem{definition}[theorem]{Definition}
 \newtheorem{remark}[theorem]{Remark}

 \newtheorem{example}[theorem]{Example}

\newtheorem{question}{Question}

\numberwithin{equation}{section}
\newcommand {\N}{\mathbb{N}} 
\newcommand {\Z}{\mathbb{Z}} 
\newcommand {\R}{\mathbb{R}} 

\newcommand{\WW}{\mathcal{W}}

 \newcommand{\Ham}{\mathrm{Ham}}

\newcommand{\diam}{\mathrm{diam}}

\DeclareMathOperator{\CA}{CA}

\DeclareMathOperator{\Sym}{Sym}
\DeclareMathOperator{\Id}{Id}

\DeclareMathOperator{\Map}{Map}

\begin{document}
\title[Strongly sofic monoids]{Strongly sofic monoids, sofic topological entropy, and surjunctivity}
\author[T.Ceccherini-Silberstein]{Tullio Ceccherini-Silberstein}
\address{Dipartimento di Ingegneria, Universit\`a del Sannio, I-82100 Benevento, Italy}
\address{Istituto Nazionale di Alta Matematica ``Francesco Severi'', I-00185 Rome, Italy}
\email{tullio.cs@sbai.uniroma1.it}
\author[M.Coornaert]{Michel Coornaert}
\address{Universit\'e de Strasbourg, CNRS, IRMA UMR 7501, F-67000 Strasbourg, France}
\email{michel.coornaert@math.unistra.fr}
\author[X.K.Phung]{Xuan Kien Phung}
\address{D\'epartement d'Informatique et de Recherche Op\'erationnelle, Universit\'e de Montr\'eal, Montr\'eal, Qu\'ebec, H3T 1J4, Canada}
\address{D\'epartement de Math\'ematiques et de Statistique, Universit\'e de Montr\'eal, Montr\'eal, Qu\'ebec, H3T 1J4, Canada}
\email{phungxuankien1@gmail.com}

\begin{abstract}
We introduce the class of strongly sofic monoids. 
This class of monoids strictly contains the class of sofic groups and is a proper subclass of the class of sofic monoids.  
We define and investigate sofic topological entropy for actions of strongly sofic monoids on compact spaces.
We show that sofic topological entropy is a topological conjugacy invariant for such actions
and use this fact to prove that every strongly sofic monoid is surjunctive.
This means that if $M$ is a strongly sofic monoid and $A$ is a finite alphabet set,
then every injective cellular automaton $\tau \colon A^M \to A^M$ is surjective.
As an application, we prove that the monoid algebra of a strongly sofic monoid with coefficients in an arbitrary field is always stably finite. Our results are extensions to strongly sofic monoids of two previously known properties of sofic groups.  
The first one is the celebrated Gromov-Weiss theorem asserting that every sofic group is surjunctive.  
The second is the Elek-Szab\'o theorem which says that group algebras of sofic groups satisfy Kaplansky's stable finiteness conjecture. 
\end{abstract}
\date{\today}
\maketitle

\setcounter{tocdepth}{1}
\tableofcontents

\section{Introduction}
Sofic groups were introduced by Gromov~\cite{gromov-esav} and Weiss~\cite{weiss-sgds}
twenty-five years ago (see for example \cite[Chapter~7]{csc-cag2} for a detailed exposition).
Roughly speaking, a sofic group is a group that can be well approximated by finite symmetric groups.
There are many groups, such as all linear groups and all residually amenable groups, which are known to be sofic, 
and also groups, such as Thompson groups $F$, $T$, and $V$, for which it is still unknown whether they are sofic or not.
In fact, the question of the existence of a non-sofic group is open up to now.
In~\cite{sofic-monoids}, the first two authors investigated a notion of soficity for monoids,
extending the one for groups in the sense that
a group is sofic if and only if it is sofic with respect to its underlying monoid structure.
For defining sofic monoids, one replaces finite symmetric groups by finite symmetric monoids
(the \emph{symmetric monoid} of a set $X$ is the set $\Map(X)$ of all self-mappings of $X$ with the composition of maps as the monoid operation).
It was shown in~\cite{sofic-monoids} that all finite monoids, all commutative monoids, all cancellative one-sided amenable monoids, and all residually finite monoids are sofic (see~\cite{kambites-large-class} for other examples of sofic monoids).
The \emph{bicyclic monoid}, that is, the monoid with presentation $B = \langle p,q : pq = 1 \rangle$, is not sofic
\cite[Theorem~5.1]{sofic-monoids}. 
The bicyclic monoid is an amenable inverse monoid.
By contrast, every amenable group is sofic (see e.g.~\cite[Proposition~7.5.6]{csc-cag2})
and, as mentioned above, no example of a non-sofic group has been found up to now.
\par
Given a finite set $A$, called the \emph{alphabet}, a \emph{cellular automaton} over a monoid $M$ is a map $\tau \colon A^M \to A^M$ which is both continuous (with respect to the prodiscrete topology on $A^M$) and $M$-equivariant (with respect to the shift action of $M$ on $A^M$).
A monoid $M$ is called \emph{surjunctive} if every injective cellular automaton with finite alphabet over $M$ is surjective 
(see~\cite{surjunctive-monoids}).
Surjunctivity for groups was first investigated by Gottschalk in~\cite{gottschalk}.
The bicyclic monoid is known to be non-surjunctive~\cite[Theorem~5.5]{surjunctive-monoids}.
No example of a non-surjunctive group has been found yet. 
Actually, it has been shown by Gromov~\cite{gromov-esav} and Weiss~\cite{weiss-sgds} that every sofic group is surjunctive
(see also \cite{phung-geometric, phung-cjm, phung-weakly, phung-gottschalk} for some generalizations),
so that an example of a non-surjunctive group would also yield an example of a non-sofic group.
\par
In the present paper, we introduce the class of \emph{strongly sofic monoids}, a proper subclass of the class of sofic monoids 
(see Definition~\ref{d:strongly-sofic-monoid} and Corollary~\ref{c:ssofic-strictly-sofic}).
A group is strongly sofic as a monoid if and only if it is sofic as a group (Proposition~\ref{p:strongly-sofic-groups-monoids}).
As every submonoid of a strongly sofic monoid is itself strongly sofic (Proposition~\ref{p:submonoids-strongly-sofic}),
this implies that every monoid that can be embedded into a sofic group is strongly sofic.
It follows that all free monoids, all cancellative commutative monoids, and all cancellative one-sided amenable monoids are strongly sofic.
However, there exist right-cancellative strongly sofic monoids that are not left-cancellative and therefore not embeddable into a group 
(Example~\ref{ex:no-group}).
The class of strongly sofic monoids  is also closed under direct products, projective limits, and inductive limits.
\par
In \cite[Question (Q6)]{surjunctive-monoids}, it was asked whether or not the Gromov-Weiss theorem on the surjunctivity of sofic groups extends to all sofic monoids. In the present paper, we shall give a partial answer to this question by establishing the following result.

\begin{theorem}
\label{t:main}
Every strongly sofic monoid is surjunctive.
\end{theorem}

One says that a monoid $M$ is \emph{directly finite} if $mm' = 1_M$ implies $m'm = 1_M$ for all $m,m' \in M$.
Equivalently, a monoid is directly finite if and only if it contains no submonoid isomorphic to the bicyclic monoid.
Given a monoid $M$ and a finite set $A$, the set $\CA(M,A)$ consisting of all cellular automata $\tau \colon A^M \to A^M$ with the composition of maps as the monoid operation, is a monoid, with the identity map $\Id_{A^M} \colon A^M \to A^M$ as the identity element.
It is known that the monoid $\CA(M,A)$ is directly finite if and only if every injective cellular automaton $\tau \colon A^M \to A^M$ is surjective (cf.~\cite[Proposition~2.3]{ccp-stabmomo}).
Thus, as a consequence of Theorem~\ref{t:main}, we get the following:

\begin{corollary}
\label{c:CA-is-df}
Let $M$ be a strongly sofic monoid and let $A$ be a finite set.
Then the monoid $\CA(M,A)$ is directly finite.
In other words, the monoid $\CA(M,A)$ contains no submonoid isomorphic to the bicyclic monoid.
\end{corollary}

A ring $R$ is called \emph{stably finite} if the multiplicative monoid  of $d \times d$ matrices with entries in $R$ is directly finite
for every integer $d \geq 1$. 
In~\cite{ccp-stabmomo}, it is shown that if $M$ is a surjunctive monoid then the monoid algebra $K[M]$ is stably finite for any field $K$.
As a consequence, we deduce from Theorem~\ref{t:main} the following result.

\begin{corollary}
\label{c:surj-monoid-stably-finite-alg}
Let $M$ be a strongly sofic monoid and let $K$ be a field.
Then the monoid algebra $K[M]$ is stably finite.
\end{corollary}

In the group setting, Corollary~\ref{c:surj-monoid-stably-finite-alg} is due to Elek and Szab{\'o}~\cite[Corollary~4.7]{es-direct}
(see also \cite[Corollary 1.4]{cc-IJM}, \cite{cscp-model}, and \cite[Corollary~8.15.8]{csc-cag2}).
The question whether the group algebra $K[G]$ is stably finite for every group $G$ and any field $K$ is known as
\emph{Kaplansky's stable finiteness conjecture} and is open up to know,
although Kaplansky proved it for fields of characteristic~$0$.
By contrast, the monoid algebra $K[M]$ is clearly not stably finite, whatever the field $K$, if $M$ is a monoid containing a submonoid isomorphic to the bicyclic monoid.

\par
In order to establish Theorem~\ref{t:main}, we shall use the same strategy as the one developed by Kerr and Li~\cite[Theorem~10.30]{kerr-li-book} for proving the Gromov-Weiss theorem on the surjunctivity of sofic groups.
As the case when $M$ is finite is trivial, we start with an infinite strongly sofic monoid $M$.
Given a strong sofic approximation $\Sigma$ of $M$ (cf.\ Proposition~\ref{p:equiv-def-strongly-sofic} for
the definition of a strong sofic approximation), we define the \emph{sofic topological entropy} $h_\Sigma(X,M)$ for any compact topological space $X$ equipped with a continuous action of $M$ and admitting a dynamically generating continuous pseudometric
(this last condition is satisfied if $X \subset A^M$ is a subshift, or $X$ is metrizable).
We then show that sofic topological entropy is a topological conjugacy invariant for actions of strongly sofic monoids 
(see Theorem~\ref{t:topo-invariant}).
It follows that if $A$ is a finite set and $\tau \colon A^M \to A^M$ is an injective cellular automaton, then
$h_\Sigma(\tau(A^M),M) = h_\Sigma(A^M,M) = \log |A|$.
Using the fact that $h_\Sigma(X,M) < \log |A|$ for any proper subshift $X \subsetneqq A^M$ (Theorem~\ref{t:monotonicity}),
we deduce that $\tau(A^M) = A^M$. This shows that $\tau$ is surjective and completes the proof of the surjunctivity of $M$.
\par
The paper is organized as follows. 
In Section~\ref{s:preliminaries}, we fix notation and review some background material on monoids, monoid actions,
shifts and subshifts over monoids, cellular automata over monoids and surjunctivity, pseudometrics, 
symmetric monoids, Hamming metrics, and sofic monoids.
In Section~\ref{s:ssmonoids}, we introduce the notion of strong soficity for monoids. 
We first establish some closure properties for the class of strongly sofic monoids.
We show that the class of strongly sofic monoids lies strictly between the class of monoids embeddable into sofic groups and the class of sofic monoids (Corollary~\ref{c:embeddable-group}, Example~\ref{ex:no-group}, Proposition~\ref{p:strongly-sofic-implies-sofic}, and Corollary~\ref{c:ssofic-strictly-sofic}).
We also show that a finite monoid is strongly sofic if and only if it is a group (Corollary~\ref{c:finite-monoid-strong-s}).
In Section~\ref{s:topo-invar}, we introduce and study the notion of sofic topological entropy for sofic monoid actions.
Given a continuous action of a sofic monoid $M$ on a compact topological space $X$, we first define the sofic topological entropy
$h_\Sigma(M,X,\rho)$ relative to a sofic approximation $\Sigma$ of $M$ and a continuous pseudometric $\rho$ on $X$.
Then we show that when $M$ is an infinite strongly sofic monoid, $\Sigma$ is a strong sofic approximation of $M$,
and $(X,M)$ admits a dynamically generating continuous pseudometric $\rho$, then the sofic topological entropy
$h_\Sigma(X,M) \coloneqq h_\Sigma(M,X,\rho)$ is in fact independent of $\rho$ (Proposition~\ref{p:h-rho-h-rho'}).
In Section~\ref{s:ste-ss}, we study sofic topological entropy for subshifts $X \subset A^M$, where $A$ is a finite set and
$M$ is an infinite strongly sofic monoid.
We show that, for any strong sofic approximation $\Sigma$ of $M$, one has 
$h_\Sigma(X,M) \leq \log |A|$, with  equality if and only if $X = A^M$ (Theorem~\ref{t:monotonicity}).
This is then used to prove Theorem~\ref{t:main}, as explained above.

\section{Preliminaries}
\label{s:preliminaries}

\subsection{Notation} 
We use the convention $\N = \{0,1, \ldots \}$ for the set of natural numbers.
The symbol $\Z$ denotes the set of integers $\{\dots,-2,-1,0,1,2,\dots\}$.
\par
The cardinality of a finite set $D$ is written $|D|$.
\par 
Given a real number $\alpha$, we denote by $\floor{\alpha} \in \Z$ (resp. $\lceil \alpha \rceil \in \Z$)
the greatest integer less (the least integer greater) than or equal to $\alpha$.
Note that $n - \lceil \alpha \rceil = \floor{n - \alpha}$ for all $n \in \Z$ and $\alpha \in \R$.
\par
Given a set $D$, we denote by $\Map(D)$ the set of all maps $f \colon D \to D$. 
We denote by $\Id_D \in \Map(D)$ the identity map of $D$, that is, $\Id_D(v) \coloneqq v$ for all $v \in D$.
Given subsets $S \subset \Map(D)$ and $D_0 \subset D$, we set
\[
S(D_0) \coloneqq  \{f(v) \colon f \in S, v \in D_0\} \subset D 
\] 
and
\begin{equation}
\label{e:inverse-set-map}
S^{-1}(D_0) \coloneqq  \bigcup_{f \in S, v \in D_0} f^{-1}(v) \subset D,
\end{equation}
where, as usual, $f^{-1}(v) \coloneqq \{u \in D: f(u) = v\} \subset D$ for all $v \in D$.
Note that if $\Id_D \in S$, then $S(D_0) \supset D_0$ (resp.\ $S^{-1}(D_0) \supset D_0$).
\par
Also, given a set $M$ and a map $\sigma \colon M \to \Map(D)$, in order to simplify notation, 
for $s \in M$ and $v \in D$ we shall simply write $\sigma(s)(v)$ instead of $(\sigma(s))(v)$.

\subsection{Monoids}
A \emph{monoid} is a semigroup admitting an identity element. 
Thus, a monoid is a set $M$ equipped with a binary  operation $M \times M \to M$, $(m_1,m_2) \mapsto m_1 m_2$,
which is \emph{associative}, i.e., such that  $(m_1 m_2) m_3 = m_1(m_2 m_3)$ for all $m_1,m_2,m_3 \in M$, and  
admits a (necessarily unique) element $1_M \in M$ such that $1_M m = m 1_M = m$ for all $m \in M$.
\par
Let $M$ be a monoid.
An element $e \in M$ such that $e^2 = e$ is called an \emph{idempotent}.
The identity element $1_M$ of $M$ is clearly an idempotent of $M$.
An idempotent $e$ of $M$ is a \emph{non-trivial idempotent} if $e \not= 1_M$.
\par
An element $a \in M$ such that $am = a$ (resp.\ $ma = a$) for all $m \in M$
is called \emph{left-absorbing} (resp.\ \emph{right-absorbing}). An element which is both left- and right-absorbing
is termed \emph{absorbing}, or a \emph{zero}. Note that there could be distinct left-absorbing (resp.\ right-absorbing)
elements, for instance if $D$ is a set, any constant map is left-absorbing in the monoid $\Map(D)$ consisting of all
maps $f \colon D \to D$ (cf.\ Section \ref{s:symm-monoids-hamming}).
On the other hand, it is clear that in any monoid there is at most one zero element.
\par
Given two monoids $M$ and $M'$, one says that a map $f \colon M \to M'$ is a \emph{monoid morphism} if
$f(1_M) = 1_{M'}$ and $f(m_1m_2) = f(m_1)f(m_2)$ for all $m_1,m_2 \in M$.
The category whose objects are monoids and arrows are monoid morphisms between them is a subcategory of the category of sets.
Note that the category of groups is a full subcategory of the category of monoids.
A \emph{submonoid} of a monoid $M$ is a subset $N \subset M$ such that $1_M \in N$ and $m m' \in N$ for all $m, m' \in N$.
If $N \subset M$ is a submonoid, then $N$ inherits from $M$ a monoid structure obtained by restricting to $N$ the monoid operation on $M$.
One says that a monoid $M$ is \emph{embeddable} into a monoid $M'$ if there exists an injective monoid morphism $M \to M'$.
This amounts to saying that there exists a submonoid of $M'$ which is isomorphic to $M$.

\subsection{Monoid actions}
An \emph{action} of a monoid $M$ on a set $X$ is a map $M \times X \to X$, $(m,x) \mapsto mx$, such that
$m_1(m_2x) = (m_1m_2)x$ and $1_M x = x$ for all $m_1,m_2 \in M$ and $x \in X$.
\par
Let $X$ be a set equipped with an action of a monoid $M$. 
One says that a subset $Y \subset X$ is \emph{$M$-invariant} if one has $my \in Y$ for all $m \in M$ and $y \in Y$.
\par
In the case when the set $X$ is equipped with a topology, one says that the action of $M$ on $X$ is \emph{continuous} if the map $X \to X$, given by $x \mapsto m x$, is continuous for every $m \in M$.
\par
If $M$ is a monoid acting on two sets $X$ and $X'$, one says that a map $f \colon X \to X'$ is \emph{$M$-equivariant} if one has 
$f(mx) = mf(x)$ for all $m \in M$ and $x \in X$.
In the case when $X$ and $X'$ are topological spaces and the actions of $M$ on $X$ and $X'$ are continuous,
one says that a map $f \colon X \to X'$ is a \emph{topological conjugacy} if $f$ is an $M$-equivariant homeomorphism.
One says that the dynamical systems $(X,M)$ and $(X',M)$ are \emph{topologically conjugate} if there exists a topological conjugacy 
$f \colon X \to X'$.

\subsection{Shifts and subshifts}
Let $M$ be a monoid, called the \emph{universe}, and let $A$ be a finite set, called the \emph{alphabet} or the \emph{set of symbols}.
Consider the set $A^M$ consisting of all maps $x \colon M \to A$.
The elements of $A^M$ are called the \emph{configurations} over the monoid $M$ and the alphabet $A$.
We equip $A^M$ with its \emph{prodiscrete topology}, i.e., the product topology obtained by taking the discrete topology on every factor $A$ of $A^M = \prod_{m \in M} A$.
Since any product of Hausdorff (resp.~totally disconnected, resp.~compact) spaces is itself
Hausdorff (resp.~totally disconnected, resp.~compact), the configuration space $A^M$ is Hausdorff, totally disconnected, and compact.  
We also equip $A^M$ with the action of $M$ defined by $(m,x) \mapsto mx  \coloneqq  x \circ R_m$
for all $m \in M$ and $x \in A^M$, where $R_m \colon M \to M$ denotes the right multiplication by $m$.
Thus, the configuration $m x$ is given by $mx(m') = x(m' m)$ for all  $m' \in M$.
This action of $M$ on $A^M$ is clearly continuous. It is called the $M$-\emph{shift}, or simply the \emph{shift}, on $A^M$.
\par
A closed $M$-invariant subset $X \subset A^M$ is called a \emph{subshift}.
\par
Let $F \subset M$ be a finite subset. A map $p \colon F \to M$, that is, an element $p \in A^F$ is called a \emph{pattern}.
The set $F$ is then called the \emph{support} of $p$.
Given a configuration $x \in A^M$, we denote by $x\vert_F \in A^F$ the \emph{restriction} of $x$ to $F$, that is, 
the pattern defined by setting $x\vert_F(m) \coloneqq x(m)$ for all $m \in F$.
Suppose that $X \subset A^M$ is a subshift. We set $X_F \coloneqq \{x\vert_F: x\in X\} \subset A^F$. 
On often refers to $X_F$ (resp.\ $A^F \setminus X_F$) 
as to the set of all $X$-\emph{admissible} (resp.\ $X$-\emph{forbidden}) patterns with \emph{support} $F$.

\subsection{Cellular automata and surjunctivity}
Let $M$ be a monoid and let $A$ be a finite set. We equip the set $A^M$ of configurations with the prodiscrete topology
and the $M$-shift. Let $X, Y \subset A^M$ be two subshifts.
A map $\tau \colon X \to Y$ which is continuous and $M$-equivariant is called a \emph{cellular automaton}.

The following algebraic characterization of cellular automata on monoids can be easily deduced from
\cite[Theorem~4.6]{surjunctive-monoids} 
(cf.\, in the group case, the classical Curtis–Hedlund–Lyndon theorem \cite[Theorem~1.8.1]{csc-cag2} and \cite[Exercise~1.41]{ECAG}).

\begin{theorem}
\label{t:CHL}
Let $M$ be a monoid, let $A$ be a finite set, and let $X,Y \subset A^M$ be two subshifts. 
Let $\tau \colon X \to Y$ be a map.
Then the following conditions are equivalent:
\begin{enumerate}[{\rm (a)}]
\item $\tau$ is a cellular automaton;
\item there exists a finite subset $S \subset M$ and a map $\mu \colon A^S \to A$
such that
\begin{equation}
\label{e:local-def-map}
\tau(x)(m) = \mu((mx)\vert_S)
\end{equation}
for all $x \in X$ and $m \in M$.
\end{enumerate}
\end{theorem}

If $S \subset M$ and $\mu \colon A^S \to A$ are as in the statement of the above theorem, one
says that $S$ is a \emph{memory set} for $\tau$ and that $\mu$ is an associated 
\emph{local defining map}. Note that if $X = A^M$, then $\mu$ is entirely determined by $\tau$ and $S$, 
since the restriction map $A^M \to A^S$ is surjective and $\mu(x\vert_S) = \tau(x)(1_M)$ 
for all $x \in A^M$ by \eqref{e:local-def-map}.
\par
Since the composition of continuous (resp.\ $M$-equivariant) maps is continuous (resp.\ $M$-equivariant),
the set $\CA(M,A)$ consisting of all cellular automata $\tau \colon A^M \to A^M$
is a monoid for the composition of maps, the identity element being the identity cellular automaton $\Id_{A^M} \colon A^M \to A^M$.
\par

The following definition was given in \cite[Definition~5.1]{surjunctive-monoids}. It extends to monoids the analogous
notion of surjunctivity for groups given by Gottschalk \cite{gottschalk} (see also \cite[Chap.~3]{csc-cag2}).

\begin{definition}
A monoid $M$ is termed \emph{surjunctive} provided the following holds: for any finite set $A$, every injective cellular automaton
$\tau \colon A^M \to A^M$ is surjective (and therefore bijective).
\end{definition}

In \cite{surjunctive-monoids} it is shown that all finite monoids, all finitely generated commutative
monoids, all cancellative commutative monoids, all residually finite monoids, all finitely
generated linear monoids, and all cancellative one-sided amenable monoids are surjunctive.
It is also shown in~\cite{surjunctive-monoids} that the bicyclic monoid $B$ and, more generally, 
all monoids containing a submonoid  isomorphic to $B$ are non-surjunctive.

\subsection{Pseudometrics}
Let $X$ be a set. A \emph{pseudometric} on $X$ is a  map $\rho \colon X \times X \to \R$ such that
$\rho(x,x) = 0$,  $\rho(x,y) = \rho(y,x)$, and $\rho(x,y) \leq \rho(x,z) + \rho(z,y)$ for all $x,y,z \in X$.
These properties imply that $\rho(x,y) \geq 0$ for all $x,y \in X$.
A pseudometric $\rho \colon X \times X \to \R$ is a \emph{metric} if, in addition,
$\rho(x,y) = 0$ implies $x = y$ for all $x,y \in X$.
\par
Given a pseudometric $\rho$ on a set $X$ and a real number $\varepsilon > 0$,
the \emph{open $(\rho,\varepsilon)$-ball} around a point $x \in X$
is the set $B_x(\varepsilon)$ consisting of all $y \in X$ such that $\rho(x,y) < \varepsilon$.
A subset $Z \subset X$ is called $(\rho,\varepsilon)$-\emph{separated} if $\rho(x,y) \geq \varepsilon$ for all distinct $x,y \in Z$.
Note that $Z \subset X$ is $(\rho,\varepsilon)$-separated if and only if $Z \cap B_z(\varepsilon) = \{z\}$ for every $z \in Z$.

\begin{lemma}
\label{l:separated-compact}
Let $X$ be a compact topological space, let $\rho \colon X \times X \to \R$ be a continuous pseudometric on $X$, and let $\varepsilon > 0$. 
Then there exists an integer $N \geq 1$ such that every $(\rho,\varepsilon)$-separated subset 
$Z \subset X$ is finite with cardinality $|Z| \leq N$.
\par
Moreover, denoting by $N_\varepsilon(X,\rho)$ the maximal cardinality of a $(\rho,\varepsilon)$-separated subset $Z \subset X$, one has
\begin{equation}
\label{e:N-eps-dec}
N_{\varepsilon'}(X,\rho) \geq N_\varepsilon(X,\rho)
\end{equation} 
for all $0 < \varepsilon' \leq \varepsilon$. 
\end{lemma}

\begin{proof}
For $x \in X$, denote by $U_x$ the set of all points $y \in X$ such that $\rho(x,y) < \varepsilon/2$.
Observe that $U_x$ is an open subset of $X$ by continuity of $\rho$, and that $x \in U_x$ since $\rho(x,x) = 0$.
Thus, the family $(U_x)_{x \in X}$ is an open cover of $X$.
By compactness of $X$, there exists a finite subset $S \subset X$ such that the subfamily $(U_x)_{x \in S}$ covers $X$.
Observe now that if $Z \subset X$ is a $(\rho,\varepsilon)$-separated subset and $x \in X$,
then there is at most one point of $Z$ in $U_x$.
This implies $|Z| \leq |S|$.
\par
The last assertion follows from the fact that if $0 < \varepsilon' \leq \varepsilon$ then every 
$(\rho,\varepsilon)$-separated subset $Z \subset X$ is, \emph{a fortiori}, $(\rho,\varepsilon')$-separated.
\end{proof}

Suppose now that $X$ is a set equipped with an action of a monoid $M$.
A pseudometric $\rho$ on $X$ is said to be \emph{dynamically generating} for $(X,M)$ if, for all distinct points $x,y \in X$,
there exists $m = m(x,y) \in M$ such that $\rho(m x, m y) > 0$.

\begin{example}
Every metric on $X$ is dynamically generating.
Indeed, if $\rho$ is a metric on $X$, then $\rho(1_Mx,1_My) = \rho(x,y) > 0$ for all distinct $x,y \in X$.
\end{example}

\begin{example}
\label{e:rho-X-rho-Y}
Suppose that $Y \subset X$ is an $M$-invariant subset.
If $\rho$ is a pseudometric on $X$ then $\rho$ induces by restriction a pseudometric $\rho_Y$ on $Y$.
Observe that every $(\rho_Y,\varepsilon)$-separated subset $Z \subset Y$ is also $(\rho,\varepsilon)$-separated.
Consequently, if $X$ is a compact topological space, $\rho$ is continuous, and $Y$ is closed in $X$,
we have
\begin{equation}
\label{e:N-epsilon-for-inv-sub}
N_\varepsilon(Y,\rho_Y) \leq N_\varepsilon(X,\rho)
\end{equation}
for all $\varepsilon > 0$.
\par
Note also that if $\rho$ is dynamically generating for $(X,M)$ then $\rho_Y$ is clearly dynamically generating for $(Y,M)$.
\par
In the sequel, we shall simply write $\rho$ instead of $\rho_Y$ if there is no risk of confusion.
\end{example}

\begin{example}
Let $M$ be a monoid and let $A$ be a finite set.
Consider the configuration space $X \coloneqq A^M$ equipped with tis prodiscrete topology and the shift action of $M$.
Then the pseudometric $\rho \colon X \times X \to \R$ defined by
\begin{equation}
\label{e:pseudo-met-shift}
\rho(x,y) \coloneqq
\begin{cases}
0 &\text{ if } x(1_M) = y(1_M) \\
1 &\text{ otherwise,}
\end{cases}
\end{equation}
for all $x,y \in X$, is dynamically generating.
Indeed, let $x,y$ be two distinct points of $X$. Then there exists $m \in M$ such that $x(m) \neq y(m)$.
As $mx(1_M)=x(m)$ and $my(1_M)=y(m)$, this implies $mx(1_M) \neq my(1_M)$ and hence $\rho(mx, my) = 1 > 0$.
\par
Observe that the pseudometric $\rho$ is continuous.
Note also that, given any $\varepsilon > 0$, a subset $Z \subset X$ is $(\rho,\varepsilon)$-separated if and only if
the evaluation map $Z \to A$, $z \mapsto z(1_M)$ is injective.
As the set of constant configurations is $(\rho,\varepsilon)$-separated, we deduce that
\begin{equation}
\label{e:N-eps-shift}
N_\varepsilon(X,\rho) = |A|
\end{equation}
for all $\varepsilon > 0$. 
\par
When $A$ has more than one element and $M$ is uncountable, the configuration space $A^M$ is not metrizable.
This shows in particular that a compact topological space equipped with a continuous monoid action that admits 
a dynamically generating continuous pseudometric may fail to be metrizable.
\end{example}

\begin{remark}
Let $X$ be a topological space equipped with a continuous action of a monoid $M$.
Suppose that there exists a dynamically generating continuous pseudometric $\rho \colon X \times X \to \R$.
Let $x$ and $y$ be distinct points of $X$. Then we can find $m \in M$ such that $\varepsilon \coloneqq \rho(mx,my) > 0$.
Then  $U \coloneqq \{z \in X: \rho(mz,mx) < \varepsilon/2\}$ and $V \coloneqq \{z \in X: \rho(mz,my) < \varepsilon/2\}$.
are disjoint open neighborhoods of $x$ and $y$, respectively. This shows that the topology on $X$ is Hausdorff.
\end{remark}

\begin{remark}
\begin{remark}
Let $X$ be a uniform space equipped with an action of a monoid $M$.
One  says that the action of $M$ on $X$ is \emph{expansive} if there exists an entourage $U_0$
of $X$ such that, for all distinct $x,y \in X$, there exists $m \in M$ such that $(mx,my) \notin U_0$.
Such an entourage $U_0$ is then called an \emph{expansiveness entourage} for the action of $M$ on $X$.
If the action of $M$ on $X$ is expansive, then there exists a continuous pseudometric on $X$
which is dynamically generating (cf.~\cite[Chapitre 9, Section 4, Th\'eor\`eme 1]{bourbaki-tg-5-10}).
Indeed, for every entourage $U$ of $X$,
there exist a continuous pseudometric $\rho = \rho_{U}$ on $X$ and a real number $\varepsilon = \varepsilon_U >0$ such that
$\{(x,y) \in X \times X : \rho(x,y) < \varepsilon\} \subset U$ (cf.~\cite[Lemma~6.11]{kelley}).
Now, if $U_0$ is an expansiveness entourage for the action of $M$ on $X$ and $x,y \in X$ are distinct,
then there exists $m \in M$ such that $(mx,my) \notin U_0$.
This implies that $\rho_{U_0}(mx,my) \geq \varepsilon_{U_0} > 0$, showing that $\rho$ is dynamically generating for $(X,M)$.
See also~\cite[Section~3]{JDCS2}.
\end{remark}
\end{remark}

\subsection{Symmetric monoids and the Hamming metric} 
\label{s:symm-monoids-hamming}
Let $D$ be a set. We denote by $\Map(D)$ the \emph{symmetric monoid} of $D$, i.e., the set consisting of all maps
$f \colon D \to D$ with the composition of maps as the monoid operation.
The identity element of the symmetric monoid $\Map(D)$ is the identity map $\Id_D \colon D \to D$.
\par
Suppose  that  $D$ is a non-empty finite set.
The \emph{Hamming metric} $d_D^\Ham$ on $\Map(D)$ is the metric defined by 
\begin{equation}
\label{e:def-hamming}
d_D^\Ham(f,g) \coloneqq \frac{1}{|D|} |\{v \in D : f(v) \not= g(v)\}|
\end{equation} 
for all $f,g \in \Map(D)$. 
\par
Note that $0 \leq  d_D^\Ham(f,g) \leq 1$ for all $f,g \in \Map(D)$.
\par

\begin{proposition}
Let $D_1, D_2, \ldots, D_n$ be a finite sequence of non-empty finite sets.  
Consider the Cartesian product $D = \prod_{i=1}^n D_i$ and the natural monoid morphism  
$\Phi \colon \prod_{i=1}^n \Map(D_i) \to \Map(D)$ given by
\[
\Phi(f)(v) = (f_1(v_1),  \ldots, f_n(v_n))
\]
for all $f = (f_i)_{i=1}^n \in \prod_{i=1}^n\Map(D_i)$ and $v = (v_i)_{i=1}^n \in D$.
Then
\begin{equation}
\label{e;hamming-product}
d_D^\Ham (\Phi(f),\Phi(g)) = 1 - \prod_{i=1}^n \left(1 - d_{D_i}^\Ham(f_i,g_i)\right)
\end{equation}
for all $f = (f_i)_{i=1}^n$ and $g = (g_i)_{i=1}^n$ in $\prod_{i=1}^n \Map(D_i)$.
\end{proposition}

\begin{proof}
The formula \eqref{e;hamming-product} immediately follows from the equality 
\[
\{v \in D: \Phi(f)(v) = \Phi(g)(v) \} = \prod_{i=1}^n \{v_i \in D_i : f_i(v_i) = g_i(v_i)\}
\]
after taking cardinalities of both sides.
\end{proof}

\subsection{Sofic monoids}
\label{subsec:sofic}

By Lemma~6.2 in~\cite{sofic-monoids},
the definition of sofic monoids proposed in~\cite[Definitions 3.1 and 3.2]{sofic-monoids} is equivalent to the following one (cf.~\cite{kambites-large-class}).

\begin{definition}
\label{d:sofic-monoid}
A monoid $M$ is called \emph{sofic} if it satisfies the following condition:
for every finite subset $K \subset M$ and every $\varepsilon > 0$,
there exist a non-empty finite set $D$ and a map $\sigma \colon M \to \Map(D)$ such that
\begin{enumerate}[{\rm (SM1)}]
\item $\sigma(1_M) = \Id_D$;
\item $d_D^\Ham(\sigma(k_1k_2),\sigma(k_1)\sigma(k_2)) \leq \varepsilon$ for all $k_1,k_2 \in K$;
\item $d_D^\Ham(\sigma(k_1),\sigma(k_2)) \geq 1 - \varepsilon$ for all distinct $k_1, k_2 \in K$.
\end{enumerate}
\end{definition}

To define sofic topological entropy, it will be more convenient to use the following characterization of soficity.

\begin{proposition}
\label{p:equiv-def-sofic}
Let $M$ be a monoid.
Then the following conditions are equivalent:
\begin{enumerate}[\rm (a)]
\item
$M$ is sofic;
\item
there exist a directed set $I$ and  a net $\Sigma = (D_i,\sigma_i)_{i \in I}$, where $D_i$ is a non-empty finite set and
$\sigma_i \colon M \to \Map(D_i)$ is a map for all $i \in I$, such that:
\begin{enumerate}[{\rm ({SA}1)}]
\item $\sigma_i(1_M) = \Id_{D_i}$ for all $i \in I$;
\item $\lim_{i \in I}  d_{D_i}^\Ham(\sigma_i(m_1m_2),\sigma_i(m_1)\sigma_i(m_2)) = 0$ for all $m_1,m_2 \in M$;
\item $\lim_{i \in I}  d_{D_i}^\Ham(\sigma_i(m_1),\sigma_i(m_2)) = 1$ for all distinct $m_1, m_2 \in M$.
\end{enumerate}
\end{enumerate}
\end{proposition}

\begin{proof}
Suppose first that $M$ satisfies (b).
Let $K$ be a finite subset of $M$ and let $\varepsilon > 0$.
Using the fact that $K$ is finite, we deduce from (SA2) and (SA3) that there exists $i_0 \in I$ such that
every $i \in I$ with $i \geq i_0$ satisfies
$d_{D_i}^\Ham(\sigma_i(k_1k_2),\sigma_i(k_1)\sigma_i(k_2)) \leq \varepsilon$ for all $k_1,k_2 \in K$, and
  $d_{D_i}^\Ham(\sigma_i(k_1),\sigma_i(k_2)) \geq 1 - \varepsilon$ for all distinct $k_1, k_2 \in K$.
As $\sigma_i(1_M) = \Id_{D_i}$ for all $i \in I$ by (SA1), 
 it follows that  $D \coloneqq D_{i_0}$ and  $\sigma \coloneqq \sigma_{i_0}$
  satisfy the conditions in Definition~\ref{d:sofic-monoid}.
This shows that  (b) implies (a).
\par
Conversely, suppose that $M$ is sofic.
Let $I$ denote the set consisting of all pairs $(K,\varepsilon)$, where $K$ is a finite subset of $M$ and $\varepsilon > 0$.
We partially order $I$ by setting $(K_1,\varepsilon_1) \leq (K_2,\varepsilon_2)$ for $(K_1,\varepsilon_1), (K_2,\varepsilon_2) \in I$ if
$K_1 \subset K_2$ and $\varepsilon_2 \leq \varepsilon_1$.
Observe that $I$ is a directed set for this partial ordering.
Indeed, for all $(K_1,\varepsilon_1), (K_2,\varepsilon_2) \in I$, the element $(K_1 \cup K_2,\min(\varepsilon_1,\varepsilon_2))$
is an upper bound for $(K_1,\varepsilon_1)$ and $(K_2,\varepsilon_2)$ in $I$.
Since $M$ is sofic, for every $i = (K,\varepsilon) \in I$, there exists a non-empty finite set $D_i$
and a map $\sigma_i \colon M \to \Map(D_i)$ such that
$\sigma_i(1_M) = \Id_{D_i}$,
$d_{D_i}^\Ham(\sigma_i(m_1m_2),\sigma_i(m_1)\sigma_i(m_2)) \leq \varepsilon$ for all $m_1,m_2 \in K$,
and $d_{D_i}^\Ham(\sigma_i(m_1),\sigma_i(m_2)) \geq 1 - \varepsilon$ for all distinct $m_1, m_2 \in K$.
Let us show that the net $\Sigma \coloneqq (D_i,\sigma_i)_{i \in I}$ satisfies the conditions in (b).
Condition (SA1) is satisfied by construction.
Let us fix now $m_1,m_2 \in M$ and $\varepsilon > 0$.
Let $i \coloneqq (K,\varepsilon)$ with $K \coloneqq \{m_1,m_2\}$.
For all $i' = (K',\varepsilon') \in I$ such that $i' \geq i$,
we have $m_1,m_2 \in K'$ and $\varepsilon' \leq \varepsilon$,
so that $d_{D_{i'}}^\Ham(\sigma_{i'}(m_1m_2),\sigma_{i'}(m_1)\sigma_{i'}(m_2)) \leq \varepsilon' \leq \varepsilon$.
Therefore $\Sigma$ satisfies (SA2).
On the other hand, if $m_1 \not= m_2$, then
$d_{D_{i'}}^\Ham(\sigma_{i'}(m_1),\sigma_{i'}(m_2))  \geq 1 - \varepsilon' \geq 1 - \varepsilon$.
This shows that $\Sigma$ also satisfies (SA3) and completes the proof that (a) implies (b).
The equivalence (a) $\iff$ (b) is established.
\end{proof}

A net $\Sigma = (D_i,\sigma_i)_{i \in I}$ as in Proposition \ref{p:equiv-def-sofic}.(b) is called a \emph{sofic approximation}
of the sofic monoid $M$.

\begin{proposition}
\label{p:sa-infinite-monoid}
Let $M$ be an infinite sofic monoid and let $\Sigma = (D_i,\sigma_i)_{i \in I}$ be a sofic approximation of $M$.
Then one has $\lim_{i \in I} |D_i| =~\infty$.
\end{proposition}

\begin{proof}
Let us fix some integer $d \geq 1$.
Since $M$ is infinite, we can find a finite subset $K \subset M$ such that $|K| = d^{d}$.
As $K$ is finite, we deduce from Condition (SA3) that there exists $i_0 \in I$ such that,
for all $i \in I$ with $i \geq i_0$, we have $d_{D_i}^\Ham(\sigma_i(k_1), \sigma_i(k_2))  > 0$
for all distinct $k_1, k_2 \in K$.
This implies that the restriction of $\sigma_i$ to $K$ is injective and hence that
$|\Map(D_i)| \geq |K| = d^d$.
As $|\Map(D_i)| = |D_i|^{|D_i|}$, we deduce that $|D_i| \geq d$ for all $i \geq i_0$.
This shows  that $\lim_{i \in I} |D_i| =~\infty$.
\end{proof}

\begin{remark}
In the case when the monoid $M$ is countable, a similar argument shows that $M$ is sofic if and only if there exists a sequence
$(D_n,\sigma_n)_{n \in \N}$, where $D_n$ is a non-empty finite set and $\sigma_n \colon M \to D_n$ is a map for all $n \in \N$,
such that $\sigma_n(1_M) = \Id_{D_n}$ for all $n \in \N$, and
\begin{align*}
&\lim_{n \to \infty}  d_{D_n}^\Ham(\sigma_n(m_1m_2),\sigma_n(m_1)\sigma_n(m_2)) = 0 \text{ for all $m_1,m_2 \in M$;} \\
&\lim_{n \to \infty}  d_{D_n}^\Ham(\sigma_n(m_1),\sigma_n(m_2)) = 1 \text{ for all distinct $m_1, m_2 \in M$}.
\end{align*}
Such a sequence is called a \emph{sofic approximation sequence} for $M$.
In the case when the sofic monoid $M$ is countably infinite, every sofic approximation sequence $(D_n,\sigma_n)_{n \in \N}$  
satisfies $\lim_{n \to \infty} |D_n| =~\infty$.
\end{remark}

\section{Strongly sofic monoids}
\label{s:ssmonoids}
While the question whether or not all sofic monoids are surjunctive remains open, we provide a positive answer 
for the subclass of strongly sofic monoids that we now introduce and study.  

\begin{definition}
\label{d:strongly-sofic-monoid}
A monoid $M$ is called \emph{strongly sofic} if it satisfies the following condition:
for every finite subset $K \subset M$ there exists an integer $\Delta_K \geq 1$ 
such that for every $\varepsilon > 0$ there exist a non-empty finite set $D$ and 
a map $\sigma \colon M \to \Map(D)$ satisfying: 
\begin{enumerate}[{\rm (SM1)}]
\item $\sigma(1_M) = \Id_D$;
\item $d_D^\Ham(\sigma(k_1k_2),\sigma(k_1)\sigma(k_2)) \leq \varepsilon$ for all $k_1,k_2 \in K$;
\item $d_D^\Ham(\sigma(k_1),\sigma(k_2)) \geq 1 - \varepsilon$ for all distinct $k_1, k_2 \in K$;
\item $|\sigma(k)^{-1}(v)| \leq \Delta_K$ for all $k \in K$ and $v \in D$. 
\end{enumerate} 
\end{definition}

We have the following immediate observations. 

\begin{proposition}
\label{p:strongly-sofic-implies-sofic}
Every strongly sofic monoid is sofic.
\end{proposition}

\begin{proof}
This is clear since Definition~\ref{d:strongly-sofic-monoid} is obtained from Definition~\ref{d:sofic-monoid}
by just adding the condition on the existence of $\Delta_K$ satisfying (SM4).
\end{proof}

\begin{proposition}
\label{p:strongly-sofic-groups-monoids}
Let $G$ be a group.
Then $G$ is sofic as a group if and only if it is strongly sofic as a monoid.
\end{proposition}
\begin{proof}
Suppose first that $G$ is sofic as a group. Thus (cf.\ \cite[Definition~7.5.2]{csc-cag2}), for every finite subset $K \subset G$ and every 
$\varepsilon >0$ there exist a non-empty finite set $D$ and a map $\sigma \colon G \to \Sym(D) \subset \Map(D)$ satisfying
(SM1), (SM2), and (SM3). As $\sigma(g)$ is bijective for all $g \in G$, Condition (SM4) is trivially satisfied for 
$\Delta_K \coloneqq 1$. This shows that $G$ is strongly sofic as a monoid.
\par
Conversely, suppose that $G$ is strongly sofic as a monoid.
This implies that $G$ is sofic as a monoid by Proposition~\ref{p:strongly-sofic-implies-sofic}.
The fact that $G$ is sofic as a group then follows from \cite[Proposition 3.4]{sofic-monoids}
(see also \cite[Definition~1.1 and Definition~1.2]{es-sofic}).
\end{proof}

\begin{proposition}
\label{p:submonoids-strongly-sofic}
Every submonoid of a strongly sofic monoid is strongly sofic.
\end{proposition}
\begin{proof}
Suppose that $M$ is a strongly sofic monoid and let $N \subset M$ be a submonoid. Let $K \subset N$ be a finite subset and let
$\varepsilon > 0$. As $M$ is  strongly sofic, there exist an integer $\Delta_K \geq 1$, a non-empty finite set 
$D$ and a map $\sigma \colon M \to \Map(D)$ satisfying (SM1) - (SM4). Denoting by $\sigma' \colon N \to \Map(D)$ the restriction
of $\sigma$ to $N$, that is, $\sigma'(m) \coloneqq \sigma(m)$ for all $m \in N$, it is clear that Conditions (SM1) - (SM4)
are satisfied verbatim by $\sigma'$. This shows that $N$ is a strongly sofic monoid.
\end{proof}

\begin{corollary}
\label{c:embeddable-group}
Every monoid that can be embedded into a sofic group is strongly sofic. 
\end{corollary}
\begin{proof}
This follows from Proposition~\ref{p:strongly-sofic-groups-monoids} and Proposition~\ref{p:submonoids-strongly-sofic}.
\end{proof}

\begin{corollary}
\label{ex:amen-sof}
All free monoids, all cancellative commutative monoids, and all cancellative one-sided amenable monoids are strongly sofic.
\end{corollary}
\begin{proof}
Let $X$ be a set and let $M \coloneqq  X^*$ denote the free monoid based on $X$.
Then $M$ is embeddable into the free group $F_X$ based on $X$.
As free groups, being residually finite, are sofic (cf.\ \cite[Corollary 7.5.11]{csc-cag2}),
it follows from Corollary~\ref{c:embeddable-group} that $M$ is strongly sofic.
\par
It is known (cf.\ \cite[Corollary~3.6]{wilde-witz}) that every cancellative left-amenable monoid is isomorphic to a submonoid of an amenable group. As the opposite semigroup of a right-amenable semigroup is left-amenable and every group is isomorphic to its opposite, we deduce that every cancellative right-amenable semigroup is also isomorphic to a submonoid of an amenable group.
Thus, the fact that every cancellative one-sided amenable monoid is strongly sofic follows from Corollary~\ref{c:embeddable-group}
since every amenable group is sofic (cf.\ \cite[Proposition 7.5.6]{csc-cag2}).
As all commutative monoids are amenable, this implies in particular that all cancellative commutative monoids are strongly sofic.
\end{proof}

\begin{proposition}
\label{p:locally-strongly-sofic}
Every locally strongly sofic monoid is strongly sofic.
\end{proposition}
\begin{proof}
Let $M$ be a locally strongly sofic monoid. Let $K \subset M$ be a finite subset.
Denote by $N$ the submonoid of $M$ generated by $K$. As $N$ is strongly sofic, there exist an integer $\Delta_K \geq 1$,
such that for every $\varepsilon > 0$ there exist
a non-empty finite set $D$ and a map $\sigma \colon N \to \Map(D)$ satisfying (SM1) - (SM4).
Let $\sigma' \colon M \to \Map(D)$ be an extension of $\sigma$, that is, $\sigma'(m) = \sigma(m)$ for all $m \in N$.
It is clear that Conditions (SM1) - (SM4) are also satisfied by $\sigma'$. This shows that $M$ is a strongly sofic monoid.
\end{proof}

\begin{proposition}
\label{p:prod-strongly-sofic}
Let $(M_i)_{i\in I}$ be a family of strongly sofic monoids. 
Then the product monoid $M \coloneqq \prod_{i \in I}M_i$ is also strongly sofic.
\end{proposition}
\begin{proof}
For each $i \in I$, let  $\pi_i \colon M \to M_i$ denote the projection  morphism.
Fix a finite subset $K \subset M$.
Then there exists a finite subset $J \subset I$ such that the projection map
$\pi_J \colon M \to M_J \coloneqq \prod_{j \in J}M_j$ is injective on $K$.
For each $j \in J$, the monoid $M_j$ is strongly sofic by our hypothesis.
Let $\Delta_{\pi_j(K)}$ be the positive integer provided by Definition~\ref{d:strongly-sofic-monoid}
(after replacing $M$ by $M_j$ and $K$ by $\pi_j(K)$).
Let now $\varepsilon > 0$. 
Choose a constant $0 < \eta < 1$ small enough so that
\begin{equation}
\label{e;eta-1}
1 - (1 - \eta)^{|J|} \leq \varepsilon
\end{equation}
and $\eta\leq \varepsilon$.
Then, for each $j \in J$, there exist a non-empty finite set $D_j$, and a map $\sigma_j \colon M_j \to \Map(D_j)$ 
satisfying Conditions (SM1) - (SM4) relative to $\pi_j(K)\subset M_j$ (instead of $K \subset M$) and $\eta$ (instead of $\varepsilon$).
Consider the positive integer $\Delta_K \coloneqq \prod_{j \in J} \Delta_{\pi_j(K)}$, the non-empty finite set 
$D \coloneqq \prod_{j \in J}D_j$, and the map $\sigma \colon M \to \Map(D)$ defined by setting
\[
\sigma(m)(v) \coloneqq (\sigma_j(m_j)(v_j))_{j \in J}
\]
for all $m = (m_i)_{i \in I} \in M$ and $v = (v_j)_{j \in J} \in D$.
\par
Since $\sigma_j$ satisfies (SM1) for all $j \in J$, we have
\[
\sigma(1_M)(v) = (\sigma_j(1_{M_j})(v_j))_{j \in J} = (v_j)_{j \in J} = v = \Id_D(v)
\]
for all $v = (v_j)_{j \in J} \in D$, that is, $\sigma(1_M) = \Id_D$, and Condition (SM1) is satisfied.
\par
Moreover, for all $k = (k_i)_{i \in I}, k' = (k_i')_{i \in I} \in K$, we have
\begin{align*}
d_D^\Ham(\sigma(kk'), \sigma(k)\sigma(k'))  & =
1 - \prod_{j \in J} \left(1 - d_{D_j}^\Ham(\sigma_j(k_jk_j'),\sigma_j(k_j)\sigma_j(k_j'))\right)
&& \text{(by \eqref{e;hamming-product})} \\ 
& \leq 1 - (1 - \eta)^{|J|}  && \text{by (SM2))} \\
& \leq \varepsilon && \text{(by \eqref{e;eta-1})},
\end{align*}
where (SM2) refers to $(D_j, \pi_j(K))$ for each $j \in J$, and $\eta$. Thus, Condition (SM2) is satisfied as well.
On the other hand, if $k$ and $k'$ are distinct elements in $K$, then there exists $j_0 \in J$ such that $k_{j_0} \neq k_{j_0}'$. 
This implies
\begin{align*}
d_D^\Ham(\sigma(k), \sigma(k'))
& = 1 - \prod_{j \in J}\left(1 - d_{D_j}^\Ham(\sigma_j(k_j), \sigma_j(k_j'))\right)
&& \text{(by \eqref{e;hamming-product})} \\ 
& \geq 1 - \left(1 - d_{D_{j_0}}^\Ham(\sigma_{j_0}(k_{j_0}),\sigma_{j_0}(k_{j_0}'))\right) \\
& \geq 1 -  \eta && \text{(by (SM2))} \\
& \geq 1 - \varepsilon && \text{(since $\eta \leq \varepsilon$)},
\end{align*}
where, (SM2) now refers to $(D_{j_0}, \pi_{j_0}(K))$ and $\eta$.
It follows that Condition (SM3) is also satisfied.
Finally, for all $k \in K$ and $v \in D$ we have
\[
|\sigma(k)^{-1}(v)| = \prod_{j \in J}|\sigma_j(k_j)^{-1}(v_j)| \leq \prod_{j \in J} \Delta_{\pi_j(K)} = \Delta_K
\]
and Condition (SM4) follows as well.
This shows that $M$ is strongly sofic.
\end{proof}

\begin{corollary}
Let $(M_i)_{i \in I}$ be a family of strongly sofic monoids.
Then their direct sum $M \coloneqq \oplus_{i \in I} M_i$ is also strongly sofic.
\end{corollary}

\begin{proof}
This immediately follows from Proposition~\ref{p:submonoids-strongly-sofic} and Proposition~\ref{p:prod-strongly-sofic}
since $M = \oplus_{i \in I} M_i$ is a submonoid of the product monoid $\prod_{i \in I} M_i$. 
\end{proof}

\begin{corollary}
\label{c:projective-strongly-sofic}
If a monoid $M$ is the limit of a projective system of strongly sofic monoids, then $M$ is strongly sofic.
\end{corollary}

\begin{proof}
If $M$ is the limit of a projective system of monoids $(M_i)_{i \in I}$, then $M$ is a submonoid of the product 
$\prod_{i \in I} M_i$. Thus, it follows from Proposition~\ref{p:submonoids-strongly-sofic} and Proposition~\ref{p:prod-strongly-sofic} 
that if every $M_i$, $i \in I$, is strongly sofic, so is $M$.
\end{proof}

Recall that an inductive system of monoids, denoted $(M_i,\psi_{ji})$, consists of the following data: a directed set $I$, 
a family $(M_i)_{i \in I}$ of monoids and, for all $i, j \in I$ such that $i \leq j$, a monoid morphism $\psi_{ji} \colon M_i \to M_j$. Moreover these morphisms must satisfy $\psi_{ii} = \Id_{M_i}$ and $\psi_{kj} \circ \psi_{ji} = \psi_{ki}$ for all $i < j < k$ in $I$. 
Then the associated limit is the monoid $M\coloneqq (\coprod_{i \in I} M_i)/\sim$ (here $\coprod$ denotes a disjoint union of sets) 
where $\sim$ is the equivalence relation on $\coprod_{i \in I} M_i$ defined as follows: for $x_i \in M_i$ and $x_j \in M_j$, $i,j \in I$, one has $x_i \sim x_j$ provided there exists $\ell \in I$ such that $i \leq \ell$ and $j \leq \ell$ and 
$\psi_{\ell i}(x_i) = \psi_{\ell j}(x_j)$ in $M_\ell$. Denoting by $[x_i] \coloneqq \{x_j \in M_j: x_i \sim x_j, j \in I\} \in M$ the equivalence class of $x_i \in M_i$, the multiplication in $M$ is defined by $[x_i][y_j] \coloneqq [\psi_{\ell i}(x_i)\psi_{\ell j}(y_j)]$, where $\ell \in I$ is such that $i \leq \ell$ and $j \leq \ell$. 
This implies that the canonical map $x_i \mapsto [x_i]$ is a monoid morphism from $M_i$ into $M$ for all $i \in I$. 

\begin{proposition}
\label{p:inductive-strongly-sofic}
If a monoid $M$ is the limit of an inductive system of strongly sofic monoids, then $M$ is strongly sofic.
\end{proposition}
\begin{proof}
Let $(M_i,\psi_{ji})$ be an inductive system of strongly sofic monoids and denote by $M$ its limit.
Let $K \subset M$ be a finite subset.
Let us set $H \coloneqq K \cup K^2$. For every $h \in H$ we can find $i = i(h) \in I$ and $h_i \in M_i$ such that
$h = [h_i]$. Let $\kappa \in I$ be such that $i(h) \leq \kappa$ for all $h \in H$.
We then set $h_{\kappa} \coloneqq \psi_{\kappa i}(h_i) \in M_{\kappa}$ for all $h \in H$.
Thus $h = [h_{\kappa}]$ for all $h \in H$.
Let now $k,k' \in K$. We have $kk' \in H$ and $kk'= [k_{\kappa}][k'_{\kappa}] = [k_{\kappa} k'_\kappa]$.
Thus we can find $j = j(k,k') \in I$ such that $\kappa \leq j$ and
$\psi_{j \kappa}((kk')_\kappa) = \psi_{j \kappa}(k_\kappa k'_\kappa)$ in $M_j$.
Let $\ell \in I$ be such that $j(k,k') \leq \ell$ for all $k,k' \in K$, and set
$h_\ell \coloneqq \psi_{\ell \kappa}(h_\kappa) \in M_\ell$ for all $h \in H$.
Again, we have $h = [h_\ell]$ for all $h \in H$.  Also note that $(kk')_\ell = k_\ell k'_\ell$ for all $k,k' \in K$. 
Finally, we set $K_\ell \coloneqq \{k_\ell: k \in K\}\subset M_\ell$.
The monoid  $M_\ell$ is strongly sofic by our hypothesis.
Thus, there exists a positive integer $\Delta_K$ such that, for every $\varepsilon >0$, there exist 
a non-empty finite set $D$ and a map $\sigma \colon M_\ell \to \Map(D)$ satisfying (SM1) - (SM4)
(with $M_\ell$ instead of $M$).
We then define a map $\overline{\sigma} \colon M \to \Map(D)$ by setting
\[
\overline{\sigma}(s) \coloneqq \begin{cases} \sigma(s_\ell) & \mbox{ if } s \in H\\
\Id_D & \mbox{ otherwise.}
\end{cases}
\]
In either case, whether $1_M \in H$ or not, we have $\overline{\sigma}(1_M) = \Id_D$, showing that $\overline{\sigma}$ satisfies (SM1).
Let now $k,k' \in K$. We then have
\[
d_D^\Ham(\overline{\sigma}(kk'),\overline{\sigma}(k)\overline{\sigma}(k')) =  d_D^\Ham(\sigma(k_\ell k'_\ell), \sigma(k_\ell)\sigma(k'_\ell)) < \varepsilon.
\]
Suppose now that $k \neq k'$. Then $k_\ell \neq k'_\ell$ so that
\[
d_D^\Ham(\overline{\sigma}(k),\overline{\sigma}(k')) =  d_D^\Ham(\sigma(k_\ell), \sigma(k'_\ell)) \geq 1 - \varepsilon.
\]
This shows that Conditions (SM2) and (SM3) are satisfied.
\par
Finally, for all $k \in K$ and $v \in D$ we have
\[
|\overline{\sigma}(k)^{-1}(v)| = |\sigma(k_\ell)^{-1}(v)| \leq \Delta_K,
\]
and Condition (SM4) follows as well.
This shows that $M$ is strongly sofic.
\end{proof}

Recall that a directed family of submonoids of a monoid $M$ is a family $(M_i)_{i \in I}$ of submonoids of $M$ such that for all $i,j \in I$ there exists $k \in I$ such that $M_i \cup M_j \subset M_k$. Note that setting $i < j$ whenever $M_i \subset M_j$, we have that 
a directed family of submonoids yields an inductive system $(M_i,\psi_{ji})$ where
$\psi_{ji}$ is the inclusion morphism $M_i \hookrightarrow M_j$ for $i < j$.

\begin{corollary}
\label{c:increasing-union-strongly-sofic}
If a monoid $M$ is the union of a directed family of strongly sofic submonoids, then $M$ is also strongly sofic.
\end{corollary}
\begin{proof}
The monoid $M$ is the limit of the associated inductive system.
\end{proof}

\begin{remark}
Since every monoid is the union of the directed family of its finitely generated submonoids, from
Corollary \ref{c:increasing-union-strongly-sofic}, we recover the fact
that every locally strongly-sofic monoid is strongly sofic (cf.\ Proposition \ref{p:locally-strongly-sofic}).
\end{remark}

Up to now, we have collected a few ``positive'' properties of the class of strongly sofic monoids.
It is now time to clarify that this class is indeed a proper subclass of the class of sofic monoids. 

\begin{proposition}
\label{p:strong-s-2}
No monoid admitting a non-trivial idempotent is strongly sofic.
\end{proposition}
\begin{proof}
Let $M$ be a monoid and suppose that there exists $a \in M$ such that $a^2 = a \neq 1_M$.
Set $K \coloneqq \{1_M, a\}$ and suppose that there exists an integer $\Delta_K \geq 1$ such that
for every $\varepsilon > 0$ there exist a non-empty finite set $D$ and a map $\sigma \colon M \to \Map(D)$ 
satisfying Conditions (SM1)-(SM4).
\par
Set $f \coloneqq \sigma(a)\in \Map(D)$, $D' \coloneqq \{v \in D: f(v) = v\} \subset D$, and
$D'' \coloneqq \{v \in D: f(v) = f^2(v)\} \subset D$. 
Then, by (SM2), we have 
\begin{equation}
\label{e:inegalite-D''}
|D''| \geq (1-\varepsilon) |D|
\end{equation} 
and, by (SM1) and (SM3), 
\begin{equation}
\label{e:inegalite-D'}
|D'| \leq \varepsilon |D|.
\end{equation} 

Moreover, if $u \in D''$ and $v \coloneqq f(u)$, then $f(v) = f^2(u) = f(u) = v$. This shows that $f(D'') \subset D'$.
Thus, by the pigeonhole principle, there exists $v \in f(D'')$ such that  $|f^{-1}(v)| \geq |D''|/|D'|$.
Consequently, from \eqref{e:inegalite-D''}, \eqref{e:inegalite-D'}, and (SM4) we deduce that $\Delta_K \geq (1 - \varepsilon)/\varepsilon$, 
which is absurd since $\varepsilon$ was arbitrary and $\lim_{\varepsilon \to 0} (1 - \varepsilon)/\varepsilon =~\infty$.
\end{proof}

\begin{corollary}
\label{c:zero-not-ss}
If a non-trivial monoid $M$ admits a one-sided absorbing element (e.g., a zero element), then $M$ is not strongly sofic.
\end{corollary}

\begin{proof}
If the monoid $M$ is not trivial and $a \in M$ is left-absorbing (resp.~right-absorbing),
then $a$ is a non-trivial idempotent of $M$.
\end{proof}

From the above discussion, we have the following:

\begin{itemize}
\item The multiplicative monoid $\{0,1\}$ is not strongly sofic.
Note that it is sofic since it is finite. 
\item 
If a set $D$ contains more than one element, then the monoid $\Map(D)$ is not strongly sofic.
Indeed, the monoid $\Map(D)$ is then non-trivial and, as mentioned above, every constant map in $\Map(D)$ is left-absorbing. 
Note that $\Map(D)$ is sofic if and only if $D$ is finite (cf.\ \cite[Proposition~4.1 and Corollary~5.5]{sofic-monoids}.
\item
The multiplicative monoid $\N$ of non-negative integers is not strongly sofic.
Note that it is sofic since it is commutative (cf.~\cite[Proposition~4.3]{sofic-monoids})
and that, by contrast, the additive monoid $\N$ is strongly sofic since it is a submonoid of the sofic group $\Z$.
\item
If $R$ is a non-trivial ring then the multiplicative monoid $R$ is not strongly sofic.
\end{itemize}

\begin{corollary}
\label{c:ssofic-strictly-sofic}
The class of strongly sofic monoids is strictly contained in the class of sofic monoids.
\end{corollary}

From Proposition~\ref{p:strong-s-2} we also deduce the following characterization of strong soficity for finite monoids.

\begin{corollary}
\label{c:finite-monoid-strong-s}
Let $M$ be a finite monoid. Then the following conditions are equivalent.
\begin{enumerate}[{\rm (a)}]
\item $M$ is a group;
\item $M$ is a strongly sofic monoid;
\item $M$ has no non-trivial idempotents.
\end{enumerate}
\end{corollary}
\begin{proof}
The implication (a) $\implies$ (b) follows from Proposition~\ref{p:strongly-sofic-groups-monoids} since every finite group is sofic.
The implication (b) $\implies$ (c) follows from Proposition~\ref{p:strong-s-2}.
Finally, the implication (c) $\implies$ (a) is a classical result from the theory of semigroups (see
\cite[Theorem 1.9 and Exercise 1]{clifford-preston}. Clifford and Preston attribute it to Frobenius).
For the convenience of the reader and the sake of completeness let us give a proof of this result.
Suppose that $M$ is not a group. Then there exists $a \in M$ which is not invertible. 
As $M$ is finite, there exist integers $1 \leq m < n$ such that $a^m = a^n$. Set $t \coloneqq n-m \geq 1$. 
Then $a^m = a^n = a^{m+t} = a^ma^t = a^{m+t}a^t = a^{m+2t} = \cdots = a^{m+mt}$.
Setting $e \coloneqq a^{mt}$ we thus have
\[
e^2 = a^{2mt} = a^{m+mt}a^{m(t-1)} = a^m a^{m(t-1)} = a^{mt} = e.
\]
As $a$ is not invertible and $mt \geq 1$, we have $a^{mt} \not= 1_M$. This shows that $e$ is a non-trivial idempotent of $M$.
\end{proof}

\begin{corollary}
Let $M$ be a strongly sofic monoid. 
Then every non-invertible element in $M$ has infinite order.
\end{corollary}

\begin{proof}
If $a$ is a non-invertible element of finite order in a monoid $M$,
then the monoid generated by $a$ is not strongly sofic by Corollary~\ref{c:finite-monoid-strong-s}.
Therefore $M$ is not strongly sofic either, by Proposition~\ref{p:submonoids-strongly-sofic}.
\end{proof}

We have seen in Corollary~\ref{c:embeddable-group} that every monoid which is embeddable into a sofic group is strongly sofic.
The following example shows that there exist strongly sofic monoids that cannot be embedded into a group.

\begin{example}[A strongly sofic monoid that cannot be embedded into a group]
\label{ex:no-group}
Consider the monoid $\Z[X]$ of all integral polynomials in one indeterminate with the monoid operation
given by the composition of polynomials.
The identity element of $\Z[X]$ is the monomial $X$.
Observe that the set $M \coloneqq \Z[X] \setminus \Z$ of all non-constant polynomials is a submonoid of $\Z[X]$.
Indeed, $X \in M$ and, if $P, Q \in M$, then $\deg (P\circ Q) = (\deg(P))(\deg(Q)) > 0$, so that $P \circ Q \in M$.
Let us show that the monoid $M$ is strongly sofic. Let $K$ be a non-empty finite subset of $M$ and set
$\Delta_K \coloneqq \max\{\deg(P): P \in K\} \geq 1$.
Let also $\varepsilon >0$.
Choose a finite field $D$ of cardinality $|D| > \Delta_K/\varepsilon$ 
and consider the map $\sigma \colon M \to \Map(D)$ which sends each polynomial $P \in M$ to the corresponding evaluation function 
$v_P \colon D \to D$ given by $v_P(a) \coloneqq P(a)$ for all $a \in D$.
It is clear that $\sigma$ is a morphism of monoids.
Therefore Conditions (SM1) and (SM2) of Definition~\ref{d:strongly-sofic-monoid} are satisfied.
Moreover, if $P, Q\in K$ are distinct, then 
\[
d^\Ham_F(v_P,v_Q) = 1 - \frac{1}{|D|} |\{a\in D: P(a)=Q(a)\}| \geq 1 -  \frac{1}{|D|} \Delta_K > 1 - \varepsilon,
\]
since the number of solutions in $D$ of the equation $P(a) = Q(a)$ is at most $\deg(P-Q) \leq \Delta_K < |D| \varepsilon$.
This shows that (SM3) is also satisfied.
Finally, for all $P \in K$, we have $0 <\deg(P) \leq \Delta_K$ and hence $|(v_P)^{-1}(a)| = |\{b \in D: P(b)=a\}| \leq \deg(P)
\leq \Delta_K$ for all $a \in D$, and Condition (SM4) follows as well. We conclude that $M$ is strongly sofic.
\par
Since $X^2 \circ (-X) = (-X)^2 = X^2 = X^2 \circ X$ but $-X \neq X$, the monoid $M$ is not left-cancellative.
Therefore the monoid $M$ cannot be embedded into a group.
\par
Remark that $M$ is right-cancellative.
Indeed, suppose that $P, Q, R \in M$ satisfy $P \circ R = Q \circ R$.
As the polynomial $R$ is not constant, it takes infinitely many values $R(x)$, $x \in \mathbb{R}$. 
Consequently, from $P(R(x)) = Q(R(x))$ for all $x \in \R$, we deduce that $P = Q$.
\par
Note also that, by suitably adapting the argument for (SM3) above, one proves that $M$ is residually finite. 
As all residually finite moinoids are surjunctive by \cite[Theorem~5.17]{surjunctive-monoids}, 
we deduce that $M$ is surjunctive (a fact that we shall alternatively deduce from our main result, namely, Theorem~\ref{t:main}).
\par
Moreover, denoting by $G$ the additive group of rational numbers (which is sofic but not residually finite), 
then $M \times G$ yields an example of a strongly sofic monoid which is not residually finite and does not embed into a group either.
\end{example}

\begin{proposition}
\label{p:equiv-def-strongly-sofic}
Let $M$ be a monoid.
Then the following conditions are equivalent:
\begin{enumerate}[\rm (a)]
\item
$M$ is strongly sofic;
\item
there exist a directed set $I$ and  a net $\Sigma = (D_i,\sigma_i)_{i \in I}$, where $D_i$ is a non-empty finite set and
$\sigma_i \colon M \to \Map(D_i)$ is a map for all $i \in I$, such that:
\begin{enumerate}[{\rm ({SA}1)}]
\item $\sigma_i(1_M) = \Id_{D_i}$ for all $i \in I$;
\item $\lim_{i \in I}  d_{D_i}^\Ham(\sigma_i(m_1m_2),\sigma_i(m_1)\sigma_i(m_2)) = 0$ for all $m_1,m_2 \in M$;
\item $\lim_{i \in I}  d_{D_i}^\Ham(\sigma_i(m_1),\sigma_i(m_2)) = 1$ for all distinct $m_1, m_2 \in M$.
\item for every finite subset $K \subset M$, there exists an integer $\Delta_K \geq 1$ such that
$|\sigma_i(k)^{-1}(v)| \leq \Delta_K$ for all $i \in I$, $k \in K$, and $v \in D_i$.
\end{enumerate}
\end{enumerate}
\end{proposition}
\begin{proof}
The proof follows the same lines as the one of Proposition~\ref{p:equiv-def-sofic} and is therefore omitted. 
\end{proof}

A net $\Sigma = (D_i,\sigma_i)_{i \in I}$ as in Proposition \ref{p:equiv-def-strongly-sofic}.(b) is called a 
\emph{strong sofic approximation} of the strongly sofic monoid $M$.
Every strong sofic approximation is trivially a sofic approximation.

\section{Sofic topological entropy for monoid actions}
\label{s:topo-invar}
Let $X$ be a compact topological space equipped with a continuous action of a sofic monoid $M$.
Suppose that $\rho$ is a continuous pseudometric on $X$.

Given a non-empty finite set $D$, consider the set $X^D = \{\varphi  \colon D \to X\}$ of all maps from $D$ to $X$.
Equip $X^D$ with the product topology and the diagonal action of $M$ given by $(m \varphi)(v) \coloneqq m \varphi(v)$ 
for all $\varphi \in X^D$ and $v \in D$. Note that the action of $M$ on $X^D$ is continuous.
Define the continuous pseudometrics $\rho_2^D$ and $\rho_\infty^D$ on $X^D$  
by setting
\begin{align*}
\rho_2^D(\varphi,\psi) &\coloneqq \left(\frac{1}{|D|} \sum_{v \in D}(\rho (\varphi(v),\psi(v)))^2 \right)^{1/2}  \text{  and} \\
\rho_\infty^D(\varphi,\psi) &\coloneqq \max_{v \in D} \rho(\varphi(v),\psi(v)) 
\end{align*}
for all $\varphi, \psi \in X^D$. 
\par
Given a finite subset $F \subset M$, a map $\sigma \colon M \to \Map(D)$,
and a real number $\delta > 0$, we set 
\begin{equation}
\label{e:def-MAP}
\Map(X,M,\rho,F,\delta,\sigma) \coloneqq \{\varphi \in X^D:
\rho_2^D(\varphi \circ \sigma(m), m \varphi) \leq \delta \mbox{ for all }m \in F\}.
\end{equation}
Observe that $\Map(X,M,\rho,F,\delta,\sigma)$ is closed in $X^D$
since $\rho_2^D$ is continuous.
Given $\varepsilon >0$, we denote by $N_\varepsilon(\Map(X,M,\rho,F,\delta,\sigma),\rho_\infty^D)$ the maximal cardinality of a 
$(\rho_\infty^D,\varepsilon)$-separated subset $Z \subset \Map(X,M,\rho,F,\delta,\sigma)$.
Observe that since $X^D$ is compact, $N_\varepsilon(\Map(X,M,\rho,F,\delta,\sigma),\rho_\infty^D)$ is finite 
by Lemma~\ref{l:separated-compact}.
\par
Let $\Sigma=(D_i, \sigma_i)_{i \in I}$ be a sofic approximation of $M$.
We successively define
\begin{align*}
h_\Sigma(X,M,\varepsilon,\rho,F,\delta) &\coloneqq
\limsup_{i \in I} \frac{1}{|D_i|} \log N_\varepsilon(\Map(X,M,\rho,F,\delta,\sigma_i),\rho_\infty^{D_i}),\\
h_\Sigma(X,M,\varepsilon, \rho, F) &\coloneqq \inf_{\delta > 0} h_{\Sigma }(X,M,\varepsilon,\rho,F,\delta), \text{  and} \\
h_\Sigma(X,M,\varepsilon,\rho) &= \inf_{F} h_\Sigma (X,M,\varepsilon,\rho,F),
\end{align*}
where the last infimum is over all finite subsets $F$ of $M$.
\par
The \emph{sofic topological entropy} of $(X,M)$ with respect to $\Sigma$ and $\rho$ is finally defined as
\begin{equation}
\label{e:entropy}
h_\Sigma(X,M,\rho) \coloneqq \sup_{\varepsilon > 0} h_\Sigma(X,M,\varepsilon,\rho).  
\end{equation}

\begin{remark}
\label{r:entropie-epsilon-limite}
Note that the map $\delta \mapsto N_\varepsilon(\Map(X,M,\rho,F,\delta,\sigma), \rho_\infty^{D_i})$ is non-decreasing.
This implies that the map $\delta \mapsto h_\Sigma(X,M,\varepsilon,\rho,F,\delta)$ is also non-decreasing and hence that
\[
h_\Sigma(X,M,\varepsilon,\rho,F) = \inf_{0 < \delta \leq \delta_0} h_\Sigma(X,M,\varepsilon,\rho,F,\delta)
\] 
for any $\delta_0 > 0$, and therefore
\[
h_\Sigma(X,M,\varepsilon,\rho,F) = \lim_{\delta \to 0} h_\Sigma(X,M,\varepsilon,\rho,F,\delta).
\]
Similarly, the map $\varepsilon \mapsto N_\varepsilon(\Map(X,M,\rho,F,\delta,\sigma), \rho_\infty^{D_i})$ is non-increasing.
This implies that the map $\varepsilon \mapsto h_\Sigma(X,M,\varepsilon,\rho)$ is also non-increasing and hence that
\[
h_\Sigma(X,M, \rho) = \sup_{0 < \varepsilon \leq \varepsilon_0} h_\Sigma(X,M,\varepsilon,\rho)
\] 
for any $\varepsilon_0 > 0$, and therefore
\begin{equation*}
h_\Sigma(X,M, \rho) = \lim_{\varepsilon \to 0} h_\Sigma(X,M,\varepsilon,\rho).  
\end{equation*}
\end{remark}

In the following, we show that, provided that the monoid $M$ is strongly sofic, $\Sigma$ is a strong sofic approximation of $M$, and
$(X,M)$ admits a dynamically generating continuous pseudometric, then
$h_\Sigma(X,M, \rho) \in \{-\infty\} \cup [0,\infty]$ does not depend on the choice of the dynamically generating continuous pseudometric 
$\rho$ for $(X,M)$ (see \cite[Proposition~10.25]{kerr-li-book} for the corresponding group theoretical original argument).
\par
We start by showing that in the definition of entropy \eqref{e:entropy} one can replace $$N_\varepsilon(\Map(X,M,\rho,F,\delta,\sigma),\rho_\infty^D),$$ which is more suitable for concrete applications, by $$N_\varepsilon(\Map(X,M,\rho,F,\delta,\sigma),\rho_2^D),$$ which
gives an easier argument for the independence we alluded to above. To prove this, we shall make use of the following application of
Stirling's approximation formula. It was widely used in \cite{kerr-li} (see \cite[Lemma~10.1 and Proposition~10.2]{kerr-li-book} and 
\cite[Lemma~A.1]{CCL2}). 
\begin{lemma} 
\label{lemma:Stirling}
Let $0 < t <1/2$. Then there exists $d_0 = d_0(t) \in \N$ and $\beta = \beta(t) \in (0, +\infty)$ with 
\begin{equation}
\label{e:stirling-2}
\lim_{t \to 0} \beta(t) = 0
\end{equation}
such that 
\begin{equation}
\label{e:stirling}
\sum_{j=0}^{\floor{t d}} {{d}\choose{j}} \leq e^{\beta d}
\end{equation}
for all $d \in \N$ such that $d \geq d_0$.
\end{lemma}

In \cite[Exercise~9.42]{GKP}, a sharpening of the above result is presented, namely,
the inequality in \eqref{e:stirling} is replaced by $\approx$, with absolute error $O(1)$.
This has a nice probabilistic interpretation in terms of large deviations theory.

The following result extends Proposition~10.23 in~\cite{kerr-li-book}.

\begin{proposition}
\label{p:ent-2-infty}
Let $X$ be a compact topological space equipped with a continuous action of a sofic monoid $M$.
Let $\Sigma = (D_i,\sigma_i)_{i \in I}$ be a sofic approximation of $M$ with $\lim_{i \in I} |D_i| =~\infty$. 
Let $\rho$ be a continuous pseudometric on $X$.
Then 
\begin{equation}
\label{e:egalite-2-infty}
h_\Sigma(X,M,\rho) = \sup_{\varepsilon > 0} \inf_{F} \inf_{\delta > 0} \limsup_{i \in I} \frac{1}{|D_i|} \log N_\varepsilon(\Map(X,M,\rho,F,\delta,\sigma_i),\rho_2^{D_i}),
\end{equation}
where the infimum $\inf_{F}$ is taken over all finite subsets $F \subset M$.
\end{proposition}
\begin{proof}
Let $D$ be a non-empty finite set. Given $\varphi, \psi \in X^D$ and setting $m \coloneqq \rho_\infty^D(\varphi,\psi)$, we have
\[
\rho_2^D(\varphi,\psi) = \left(\frac{1}{|D|} \sum_{v \in D}(\rho (\varphi(v),\psi(v)))^2 \right)^{1/2} \leq
\left(\frac{1}{|D|} \sum_{v \in D}m^2 \right)^{1/2} = m = \rho_\infty^D(\varphi,\psi).
\]
This shows that $\rho_2^D \leq \rho_\infty^D$. Consequently, given $\varepsilon > 0$, we have that any 
$(\rho_2^D,\varepsilon)$-separated subset $Z \subset X^D$ is $(\rho_\infty^D,\varepsilon)$-separated. 
We deduce that, given a finite subset $F \subset M$, $\delta > 0$, and $\sigma \colon M \to \Map(D)$, the inequality
$$N_\varepsilon(\Map(X,M,\rho,F,\delta,\sigma),\rho_2^D) \leq N_\varepsilon(\Map(X,M,\rho,F,\delta,\sigma),\rho_\infty^D)$$ holds.
Thus, in \eqref{e:egalite-2-infty} the left hand side is greater than or equal to the right hand side. Let us show the reverse inequality.
\par
Let $\varepsilon, \kappa > 0$ with $\varepsilon \leq \kappa < 1/2$, and let $D$ be a non-empty finite set.
Fix $\varphi \in X^D$.
Suppose that $\psi \in X^D$ is in the open $(\rho_2^D,\varepsilon^2)$-ball around $\varphi$,
that is,
\begin{equation}
\label{e:distance-psi-phi}
\sum_{v \in D} \rho(\varphi(v),\psi(v))^2 < \varepsilon^4 |D|.
\end{equation}
Denoting by $D_\psi' \subset D$ the set of all $v \in D$ such that $\rho(\varphi(v),\psi(v)) \geq \varepsilon$,
we see from \eqref{e:distance-psi-phi} that
\begin{equation}
\label{e:d-prime-psi}
|D_\psi'|\leq  \varepsilon^2 |D|.
\end{equation}
It follows that the set
$D_\psi \coloneqq \{v \in D: \rho(\varphi(v),\psi(v)) < \varepsilon\} = D \setminus D_\psi'$ satisfies
$|D_\psi| = |D \setminus D_\psi'| = |D| - |D_\psi'| \geq (1 - \varepsilon^2)|D|$, so that
\begin{equation}
\label{e:d-psi}
|D_\psi| \geq \lceil(1 - \varepsilon^2)|D|\rceil.
\end{equation}
Let $Z_\varphi \subset X^D$ be a $(\rho_\infty^D, 2 \kappa)$-separated subset of the open $(\rho_2^D,\varepsilon^2)$-ball around $\varphi$.
We claim that
\begin{equation}
\label{e:upper-bound}
|Z_\varphi| \leq  \sum_{j=\lceil(1 - \varepsilon^2)|D|\rceil}^{|D|} {|D| \choose{j}} \cdot  N_\kappa(X,\rho)^{\varepsilon^2 |D|}.
\end{equation}
Indeed, let us fix a maximal $(\rho,\kappa)$-separated subset $U$ of $X$, so that $|U| = N_\kappa(X,\rho)$.
Then, for every $x \in X$, there exists $u_x \in U$ such that $\rho(u_x,x) < \kappa$.
Given $\psi \in Z_\varphi$, define a map $f_\psi \colon D_\psi' \to U$ by setting
$f_\psi(v) \coloneqq u_{\psi(v)}$ for all $v \in D_\psi'$.
Let $\psi_1, \psi_2 \in Z_\varphi$ with $\psi_1 \not= \psi_2$ and $D_{\psi_1} = D_{\psi_2}$.
Let us show that $f_{\psi_1} \not= f_{\psi_2}$.
This will establish~\eqref{e:upper-bound} since every $\psi \in Z_\varphi$ satisfies
$|D_\psi| \geq \floor{(1 - \varepsilon^2)|D|}$ and $|D_\psi'| \leq \varepsilon^ 2|D|$
by~\eqref{e:d-psi} and~\eqref{e:d-prime-psi}, respectively.
Set $D_0 \coloneqq D_{\psi_1} = D_{\psi_2}$.
Then, for all $v \in D_0$, we have $\rho(\psi_1(v), \psi_2(v)) \leq \rho(\psi_1(v), \varphi(v)) +
\rho(\psi_2(v), \varphi(v)) < \varepsilon + \varepsilon = 2 \varepsilon \leq 2 \kappa$.
As $\rho_\infty^D(\psi_1,\psi_2) \geq 2 \kappa$, we deduce that there must exist
$v_0 \in D \setminus D_0$ such that $\rho(\psi_1(v_0), \psi_2(v_0)) \geq 2 \kappa$.
It follows that $f_{\psi_1}(v_0) = u_{\psi_1(v_0)} \neq u_{\psi_2(v_0)} = f_{\psi_2}(v_0)$.
Therefore, $f_{\psi_1} \neq f_{\psi_2}$.
This completes the proof of~\eqref{e:upper-bound}.
\par
By virtue of Lemma~\ref{lemma:Stirling}, there exist $d_0 = d_0(\varepsilon) \in \N$ and $\beta = \beta(\varepsilon) > 0$ such that
$\lim_{\varepsilon \to 0}\beta(\varepsilon)  = 0$ and, provided that $|D| \geq d_0$,
\[
 \sum_{j=\lceil(1 - \varepsilon^2)|D|\rceil}^{|D|} {|D| \choose{j}} 
= \sum_{j=0}^{\floor{\varepsilon^2|D|}} {|D| \choose{j}} \leq e^{\beta |D|}.
\]
From \eqref{e:upper-bound} we deduce that
\begin{equation}
\label{e:upper-bound-2}
|Z_\varphi| \leq e^{\beta |D|} N_\kappa(X,\rho)^{\varepsilon^2 |D|}
\end{equation}
provided that $|D| \geq d_0$.
\par
In order to complete the proof of the proposition,
let $F$ be a finite subset of $M$, let $\delta > 0$, and let $\sigma \colon M \to \Map(D)$.
Suppose that $Z_2$ is a maximal $(\rho_2^D,\varepsilon^2)$-separated subset of $\Map(X,M,\rho,F,\delta,\sigma)$.
Then, for every $\psi \in \Map(X,M,\rho,F,\delta,\sigma)$, there exists some
$\varphi \in Z_2$ such that $\rho_2^D(\varphi,\psi) < \varepsilon^2$.
In particular, if $Z_\infty$ is a maximal $(\rho_\infty^D, 2 \kappa)$-separated subset of $\Map(X,M,\rho,F,\delta,\sigma)$,
then every $\psi \in Z_\infty$ is contained in the open $(\rho_2,\varepsilon^2)$-ball around some $\varphi \in Z_2$.
Since any such ball contains at most $|Z_\varphi|$ elements of $Z_\infty$, we deduce from \eqref{e:upper-bound-2} that
\begin{equation}
\label{e:inegalite-finale}
N_{2 \kappa}(\Map(X,M,\rho,F,\delta,\sigma),\rho_\infty^D) 
\leq e^{\beta |D|} N_\kappa(X,\rho)^{\varepsilon^2 |D|}
N_{\varepsilon^2}(\Map(X,M,\rho,F,\delta,\sigma),\rho_2^D)
\end{equation}
provided that $|D| \geq d_0$.
By replacing respectively $D$ and $\sigma$ by $D_i$ and $\sigma_i$ for $i \in I$ large enough,
then, taking logarithms, dividing by $|D_i|$, and taking, in order, $\limsup_{i \in I}$, $\inf_{\delta > 0}$, $\inf_F$,
and (cf.\ Remark~\ref{r:entropie-epsilon-limite}) $\lim_{\varepsilon \to 0}$, and $\lim_{\kappa \to 0}$,
we deduce from \eqref{e:inegalite-finale} that, for the sofic approximation $\Sigma = (D_i,\sigma_i)_{i \in I}$ of $M$,
the right hand side in \eqref{e:egalite-2-infty} is greater than or equal to the left hand side,
so that they are in fact equal.
\end{proof}
\begin{lemma}
\label{l:maps-1}
Let $X$ be a compact topological space equipped with a continuous action of a monoid $M$.
Let $\rho$ and $\rho'$ be continuous pseudometrics on $X$ and suppose that $\rho'$ is dynamically generating for $(X,M)$.  
Let $\delta > 0$. Then there exist a finite set $F'' \subset M$ and $\delta'' > 0$ such that if 
$x,y \in X$ satisfy $\rho'(sx,sy) < \delta''$ for all $s \in F''$, then $\rho(x,y) < \delta$.  
\end{lemma}

\begin{proof}
Consider the open subset $V \subset X \times X$ defined by
\[
V \coloneqq \{(x,y) \in X \times X : \rho(x,y) < \delta\}.
\]
Since $\rho'$ is dynamically generating, given $(x,y) \in (X \times X) \setminus V$, we can find
$s_{x,y} \in M$ such that $d_{x,y}\coloneqq \rho'(s_{x,y} x, s_{x,y} y) > 0$.
By continuity of $\rho'$, for every $(x,y) \in (X \times X) \setminus V$
the set
\[
W_{x,y}\coloneqq \{(z,t) \in (X \times X) \setminus V: \rho'(s_{x,y} z, s_{x,y} t) > d_{x,y}/2\}
\]
is an open neighborhood of $(x,y)$ in $(X \times X) \setminus V$.
We thus obtain an open covering
\[
(X \times X)\setminus V = \bigcup_{(x,y) \in (X \times X)\setminus V} W_{x,y}.
\]
Since $(X \times X) \setminus V$ is compact, there exists a finite subset $E \subset (X \times X) \setminus V$ such that
\[
(X \times X) \setminus V = \bigcup_{(x,y) \in E} W_{x,y}.
\]
Set $F'' \coloneqq \{s_{x,y}: (x,y) \in E\}$ and $\delta'' \coloneqq \min_{(x,y) \in E} d_{x,y}/2$.
Note that $F''$ is a finite subset of $M$ and that $\delta'' > 0$.
We claim that $F''$ and $\delta''$ have the required properties.
Indeed, consider the subset $V'' \subset X \times X$ defined by
\[
V'' \coloneqq \{(x,y) \in X \times X : \rho'(sx,sy) < \delta'' \text{ for all } s \in F''\}.
\]
We have to show that $V'' \subset V$.
Let $(x,y) \in V''$. Then $\rho'(s_{z,t} x, s_{z,t} y) < \delta'' \leq d_{z,t}/2$ for all $(z,t) \in E$,
so that $(x,y) \notin  \bigcup_{(z,t) \in E} W_{z,t}= X\setminus V$ and hence $(x,y) \in V$. This proves our claim and finishes the proof.
\end{proof}

The following result is an extension of Lemma~10.24 in~\cite{kerr-li-book}.

\begin{lemma}
\label{l:maps}
Let $X$ be a compact topological space equipped with a continuous action of a strongly sofic monoid $M$.
Let $\Sigma = (D_i,\sigma_i)_{i \in I}$ be a strong sofic approximation of $M$.
Let $\rho$ and $\rho'$ be two continuous pseudometrics on $X$ and suppose that $\rho'$ is
dynamically generating for $(X,M)$. Let $F \subset M$ be a finite subset and let $\delta > 0$.
Then there are a finite subset $F' \subset M$, $\delta' > 0$, and $i_0 \in I$ such that
\begin{equation}
\label{e:inclusion-map}
\Map(X,M,\rho',F',\delta', \sigma) \subset \Map(X,M,\rho,F,\delta,\sigma)
\end{equation}
for all $\sigma = \sigma_i$ with $i \geq i_0$.
\end{lemma}
\begin{proof}
Since $X \times X$ is compact and $\rho$ is continuous, we have $\diam(X,\rho) \coloneqq \sup_{x,y \in X}\rho(x,y) \in [0,\infty)$.
Choose some real number $\delta_0 > \diam(X,\rho)$.
By Lemma~\ref{l:maps-1}, there exist a finite subset $F'' \subset M$ and a constant $\delta'' > 0$ such that if $x,y \in X$ 
satisfy $\rho'(sx,sy) < \delta''$ for all $s \in F''$, then $\rho(x,y) < \delta/2$. 
Consider the finite subset 
\[
F' \coloneqq F'' \cup F \cup (F''F) \subset M.
\]
As $\Sigma$ is a strong sofic approximation, there exists an integer $\Delta_{F'} \geq 1$ which satisfies the Conditions 
(SA1)-(SA4) for $K=F'$. 
Set 
\[
\delta' \coloneqq \min\{(\delta''/2)^2, ((\Delta_{F'} +2) |F''||F|  (1+\delta_0^2))^{-1}\} > 0. 
\]
Let $i_0 \in I$ such that setting $D \coloneqq D_i$ and $\sigma \coloneqq \sigma_i$ for all $i \geq i_0$, Conditions
(SM1)-(SM4) are satisfied for $\varepsilon \coloneqq \delta'$ and $K \coloneqq F'$.  
Let $\varphi \in \Map(X,M,\rho',F',\delta',\sigma)$ and define 
\begin{align*}
A& \coloneqq \{v \in D: \rho'(\varphi (\sigma(st)(v)), st \varphi(v)) < \delta''/2 \text{ for all } s\in F'' \text{ and } t \in F\}\\
B& \coloneqq\{v \in D: \rho'(\varphi (\sigma(s)\sigma(t)(v)), s \varphi(\sigma(t)(v))) < \delta''/2 \text{ for all } s\in F'' \text{ and } t \in F\}.
\end{align*}
By summing up the inequalities $\left((\rho')_2^D(\varphi \circ \sigma(st), st \varphi)\right)^2 < (\delta')^2$ over all $s \in F''$ and $t \in F$, from the definition of $A$ and the relation $(\delta'')^2 \geq 4\delta'$ we deduce that 
\begin{align*}
|F''|\cdot|F| (\delta')^2 & > \sum_{s\in F''}\sum_{t \in F}  ((\rho')_2^D(\varphi \circ \sigma(st)), st  \varphi))^2\\
& = 
 \frac{ 1}{|D|} \sum_{v \in D}   \sum_{s\in F''}\sum_{t \in F} (\rho'(\varphi ( \sigma(st)(v)), st \varphi(v)))^2 \\
& \geq   \frac{ 1}{|D|} \sum_{v \in D\setminus A} \sum_{s\in F''}\sum_{t \in F} (\rho'(\varphi (\sigma(st)(v)), st \varphi(v)))^2 \\
& \geq  \frac{|D\setminus A| }{|D|}   (\delta''/2)^2 
 \geq  \frac{|D\setminus A| }{|D|} \delta'.
\end{align*}
Therefore, 
$|D\setminus A| < |F''|\cdot|F|\cdot|D|\delta'$, and we obtain 
\begin{align}
\label{e:maps-1}
|A| = |D|  -  |D\setminus A| > (1-|F''|\cdot|F|\delta')|D|. 
\end{align}
Now, since $F\subset F' = K$, Condition (SM4) implies that 
$|\sigma(t)^{-1}(u)| \leq \Delta_{F'}$ for all $t \in F$ and $u \in D$. Consequently, for all $s \in F''$ and $t \in F$ we have 
\begin{align*}
\Delta_{F'} (\delta')^2 & > \Delta_{F'}((\rho')_2^D(\varphi \circ \sigma(s)), s \varphi))^2 \\ 
& =  \frac{1}{|D|}   \sum_{u \in D}   \Delta_{F'} (\rho'(\varphi ( \sigma(s)(u)), s \varphi(u))))^2 \\
& 
\geq \frac{1}{|D|}   \sum_{u \in D} \sum_{ v \in \sigma(t)^{-1}(u)}   (\rho'(\varphi ( \sigma(s)\sigma(t)(v)), s \varphi(\sigma(t)(v))))^2 \\
& = \frac{1}{|D|}   \sum_{v \in D}    (\rho'(\varphi ( \sigma(s)\sigma(t)(v)), s \varphi(\sigma(t)(v))))^2.
\end{align*}
By summing up the above inequalities over all $s \in F''$ and $t \in F$, we deduce that 
\begin{align*}
\Delta_{F'} \cdot |F''|\cdot |F| (\delta')^2 & > \frac{1}{|D|} \sum_{v \in D} \sum_{s\in F''}\sum_{t \in F}  
(\rho'(\varphi (\sigma(s)\sigma(t)(v)), s \varphi(\sigma(t)(v))))^2  \\
& \geq  \frac{1}{|D|} \sum_{v \in D\setminus B} \sum_{s\in F''}\sum_{t \in F} (\rho'(\varphi (\sigma(st)(v)), st \varphi(v)))^2 \\
& \geq  \frac{|D\setminus B|}{|D|} (\delta''/2)^2 
\geq  \frac{ |D\setminus B|}{|D|} \delta',
\end{align*} 
whence $|D\setminus B| < \Delta_{F'} \cdot |F''|\cdot |F|\cdot |D|\delta'$ and 
\begin{align}
\label{e:maps-2}
|B| = |D| -  |D\setminus B| > (1-\Delta_{F'} |F''|\cdot |F|\delta')|D|. 
\end{align}
Combining \eqref{e:maps-1} and \eqref{e:maps-2}, we obtain 
\begin{align}
\label{e:maps-3}
|A \cap B| = |A + |B| - |A \cup B| \geq |A| + |B| - |D| > (1- (1 + \Delta_{F'}) |F''|\cdot |F|\delta')|D|. 
\end{align}
Using Condition (SM2), a similar argument shows that for 
\[
C \coloneqq \{v \in D\: \sigma(s)\sigma(t)(v) = \sigma(st)(v) \text{ for all } s\in F'' \text{ and } t \in F\}, 
\]
one has   
\begin{align*}
|F''|\cdot |F| \delta' |D|
& = \sum_{s\in F''}\sum_{t \in F} |D|\delta' \\
& >  \sum_{s\in F''}\sum_{t \in F} |D| d^\Ham_D(\sigma(s) \sigma(t), \sigma(st)) \\
& =  \sum_{s\in F''}\sum_{t \in F} |\{v  \in D : \sigma(s)\sigma(t)(v) \neq  \sigma(st)(v) \}| \\
& \geq |D\setminus C|.
\end{align*}
Combining with \eqref{e:maps-3}, we obtain 
\begin{align}
\label{e:maps-4}
|A\cap B \cap C| \geq |A\cap B| - |D\setminus C| \geq (1- (\Delta_{F'}+2) |F''|\cdot |F|\delta')|D|. 
\end{align} 
Observe now that if $v\in A\cap B\cap C$ 
and $t \in F$, then for all $s\in F''$, one has 
\begin{align*}
\rho'(s \varphi(\sigma(t)(v)), st\varphi(v))
& \leq  \rho'(s \varphi(\sigma(t)(v)), \varphi(\sigma(st)(v))) + \rho'(\varphi(\sigma(st)(v)), s t\varphi(v))\\
& =_*     \rho'(s \varphi(\sigma(t)(v)), \varphi(\sigma(s)\sigma(t)(v)))+ \rho'(\varphi(\sigma(st)(v)), s t\varphi(v))\\
& <_* \delta''/2 + \delta''/2 \\
& =\delta'',
\end{align*}
where $=_*$ follows from the fact that $v \in C$ and $<_*$ follows from the fact that $v \in B$ and $v \in A$.
Thus, by the choices of $F''$ and $\delta''$,
\[
\rho(\varphi(\sigma(t)(v)), t\varphi(v)) < \delta/2.
\] 
Finally, by using \eqref{e:maps-4} and by the choice of $\delta'$, and recalling that 
$\delta_0 = \mathrm{diam}(X,\rho)$, we deduce that for all $t \in F$
\begin{align*}
(\rho^D_2 ( \varphi(\sigma(t)), t \varphi))^2 & =  \frac{1}{|D|}\sum_{v \in D} (\rho(\varphi(\sigma(t)(v)), t\varphi(v))^2\\
& = \frac{1}{|D|}\sum_{v\in D\setminus(A\cap B\cap C)} (\rho(\varphi(\sigma(t)(v)), t\varphi(v))^2\\ 
& \ \ \ \ \qquad \qquad + \frac{1}{|D|}\sum_{v\in A\cap B\cap C} (\rho(\varphi(\sigma(t)(v)), t\varphi(v))^2 \\
& \leq \frac{|D\setminus(A\cap B\cap C)|}{|D|} \delta_0^2 + \frac{|A\cap B\cap C|}{|D|} (\delta/2)^2\\
&\leq  (\Delta_{F'} +2) |F''|\cdot |F|\delta' \delta_0^2 + \delta^2/4\\
  & \leq 3 \delta^2/4+ \delta^2/4 = \delta^2. 
\end{align*}
This shows that $\varphi \in \Map(X,M,\rho,F,\delta,\sigma)$, and the proof is complete. 
\end{proof}

The following result extends Proposition~10.25 in~\cite{kerr-li-book}.

\begin{proposition}
\label{p:h-rho-h-rho'}
Let $X$ be a compact topological space equipped with a continuous action of an infinite strongly sofic monoid $M$.
Let $\Sigma = (D_i,\sigma_i)_{i \in I}$ be a strong sofic approximation of $M$.
Suppose that $(X,M)$ admits dynamically generating continuous pseudometrics $\rho$ and $\rho'$. 
Then $h_\Sigma(X,M,\rho) = h_\Sigma(X,M,\rho')$. 
\end{proposition}

\begin{proof}  
Let $\varepsilon>0$ and denote by $\delta'_0 = \mathrm{diam}(X,\rho')$ the $\rho'$-diameter of $X$. 
By Lemma~\ref{l:maps-1}, we can find a finite subset $E \subset M$ and $\eta > 0$ such that 
if $x,y \in X$ satisfy $\rho(sx,sy) < \sqrt{3\eta}$ for all $s \in E$, then $\rho'(x,y) < \varepsilon/\sqrt{2}$. 
Up to taking a smaller $\eta$, if necessary, we may also suppose that
\begin{align}
\label{e:invariant-1}
3\eta |E| (\delta'_0)^2 < \varepsilon^2/2.
\end{align}

Fix a finite subset $F\subset M$ and $\delta>0$ such that $E\subset F$ and $\delta < \eta$. 
By Lemma~\ref{l:maps}, there exists a finite subset $F' \subset M$, $\delta'>0$, and $i_0 \in I$ such that 
$\Map(X,M,\rho', F', \delta', \sigma_i) \subset \Map(X,M,\rho, F, \delta, \sigma_i)$ for all $i \in I$ with $i \geq i_0$. 
\par 
Let $i \in I$ with $i \geq i_0$. Set $D \coloneqq D_{i}$ and $\sigma \coloneqq\sigma_{i}$.
Let $\varphi_1, \varphi_2 \in \Map(X,M,\rho', F', \delta', \sigma)$ and suppose that $\rho^{D}_\infty(\varphi_1, \varphi_2) < \eta$. 
We claim that $(\rho')_2^D(\varphi_1, \varphi_2) < \varepsilon$. 
Indeed, for all $s \in E$ we have
\begin{align}
\label{e:invariant-3}
    \rho^D_2(s\varphi_1, s\varphi_2) & \leq \rho^D_2(s\varphi_1, \varphi_1\circ \sigma(s)) + \rho^D_2(\varphi_1\circ \sigma(s), \varphi_2\circ \sigma(s)) + \rho^D_2(\varphi_2\circ \sigma(s), s\varphi_2) \\
    & \leq \delta + \rho^D_\infty(\varphi_1\circ \sigma(s), \varphi_2\circ \sigma(s)) + \delta \nonumber\\
& \leq \delta + \eta+ \delta < 3\eta.  \nonumber 
\end{align}
From the choice of $E$ we have that
\begin{align*}
Q  \coloneqq \{v \in D \colon \rho(s\varphi_1(v), s\varphi_2(v)) < \sqrt{3\eta} \text{ for all } s \in E\}  \subset  \{v \in D \colon \rho'(\varphi_1(v), \varphi_2(v)) < \varepsilon/\sqrt{2} \}.
\end{align*} 
Moreover, by \eqref{e:invariant-3}, 
\begin{align*}
 |D| \cdot |E| (3\eta)^2  \geq \sum_{v \in D} \sum_{s \in E} ( \rho(s\varphi_1(v), s\varphi_2(v)) )^2   \geq \sum_{v \in D\setminus Q} \sum_{s \in E} (\rho(s\varphi_1(v), s\varphi_2(v)))^2  \geq |D \setminus Q|3\eta 
\end{align*}
so that $|D\setminus Q| \leq  |D| \cdot |E| 3 \eta$.  Consequently, combining with \eqref{e:invariant-1}, we obtain 
\begin{align*}
((\rho')_2^D(\varphi_1, \varphi_2))^2  & = \frac{1}{|D|}\sum_{v \in D} (\rho'(\varphi_1(v), \varphi_2(v)))^2\\
& = \frac{1}{|D|}\sum_{v \in D \setminus Q} (\rho'(\varphi_1(v), \varphi_2(v)))^2 + \frac{1}{|D|}\sum_{v \in Q} (\rho'(\varphi_1(v), \varphi_2(v)))^2 \\
& \leq \frac{|D \setminus Q|}{|D|}  (\delta'_0)^2 +  \frac{|Q|}{|D|} \varepsilon^2/2 \\
& \leq |E| 3 \eta (\delta'_0)^2 + \varepsilon^2/2\\
& \leq \varepsilon^2/2+ \varepsilon^2/2=\varepsilon^2,
\end{align*}
so that $(\rho')_2^D(\varphi_1, \varphi_2) < \varepsilon$, and the claim follows.
From the claim, we immediately deduce 
\begin{align}
\label{e:invariant-2}
N_\varepsilon(\Map(X,M,\rho', F', \delta', \sigma), (\rho')^D_2) \leq N_{\eta} (\Map(X,M,\rho, F, \delta, \sigma), \rho^D_\infty). 
\end{align} 
It follows that 
\begin{multline*}
\limsup_{i \in I}   \frac{1}{|D_i|} \log N_\varepsilon(\Map(X,M,\rho', F', \delta', \sigma_i), (\rho')^{D_i}_2) \\
\leq   \limsup_{i \in I} \frac{1}{|D_i|} \log  N_{\eta} (\Map(X,M,\rho, F, \delta, \sigma_i), \rho^{D_i}_\infty). 
\end{multline*}
Since $F \subset M$ was arbitrary (with $E \subset F$) and $\delta>0$ was arbitrary (with $\delta<\eta$), 
from the above inequality we deduce  that 
\begin{multline*}
\inf_{F'} \inf_{\delta'}  \limsup_{i \in I}  \frac{1}{|D_i|} \log N_\varepsilon(\Map(X,M,\rho', F', \delta', \sigma_i), (\rho')^{D_i}_2)\\ \leq \inf_{F} \inf_{\delta}  \limsup_{i \in I}  \frac{1}{|D_i|} \log N_\eta (\Map(X,M,\rho, F, \delta, \sigma_i), \rho^{D_i}_2).   
\end{multline*}
Finally, by passing to the limit for $\varepsilon\to 0$, so that also $\eta\to 0$ (cf.\ \eqref{e:invariant-1}), 
we conclude from Proposition~\ref{p:sa-infinite-monoid} and Proposition~\ref{p:ent-2-infty} that 
$h_\Sigma(X,M, \rho') \leq  h_\Sigma(X,M, \rho)$. 
As $\rho$ and $\rho'$ play symmetric roles, we deduce that $h_\Sigma(X,M, \rho') = h_\Sigma(X,M, \rho)$. 
\end{proof}

Let $X$ be a compact space equipped with a continuous action of an infinite strongly sofic monoid $M$ and admitting a dynamically generating continuous pseudometric $\rho$.
Given a strong sofic approximation $\Sigma = (D_i,\sigma_i)_{i \in I}$ of $M$,
we define the \emph{sofic topological entropy} of $(X,M)$ with respect to $\Sigma$ by setting $h_\Sigma(X,M) \coloneqq h_\Sigma(X,M,\rho)$.
The fact that $h_\Sigma(X,M)$ does not depend on the choice of the dynamically generating continuous pseudometric $\rho$ follows from Proposition~\ref{p:h-rho-h-rho'}.

\begin{theorem}
\label{t:topo-invariant}
Let $X$ and $Y$ be compact topological spaces equipped with continuous actions of an infinite strongly sofic monoid $M$.
Let $\Sigma = (D_i,\sigma_i)_{i \in I}$ be a strong sofic approximation of $M$.
Suppose that the dynamical systems $(X,M)$ and $(Y,M)$ are topologically conjugate
and that $(X,M)$ admits a dynamically generating continuous pseudometric.
Then $(Y,M)$ admits a dynamically generating continuous pseudometric and one has $h_\Sigma(X,M) = h_\Sigma(Y,M)$ for all strong sofic approximations $\Sigma$ of $M$.
\end{theorem}

\begin{proof}
Since $(X,M)$ and $(Y,M)$ are topologically conjugate, there exists an $M$-equivariant homeomorphism $f \colon X \to Y$.
Let $\rho$ be a dynamically generating continuous pseudometric for $(X,M)$.
Then it is straightforward that the map $\rho' \coloneqq \rho \circ (f^{-1} \times f^{-1}) \colon Y \times Y \to \R$ is a dynamically generating continuous pseudometric for $(Y,M)$.
\par
Let now $F \subset M$ be a finite subset, let $\varepsilon,\delta > 0$, let $D$ be a non-empty finite set, 
and let $\sigma \colon X \to \Map(D)$.
Consider the map $\sigma' \coloneqq \sigma \circ f^{-1} \colon Y \to \Map(D)$.
For all $\varphi \in \Map(X,M,\rho, F, \delta, \sigma)$ set $\varphi' \coloneqq f \circ \varphi$. It is straightforward that
$\varphi' \in \Map(Y,M, \rho', F, \delta, \sigma')$ and that if also $\psi \in \Map(X,M,\rho, F, \delta, \sigma)$ then
$\varphi$ and $\psi$ are $(\rho_\infty^D,\varepsilon)$-separated if and only if $\varphi',\psi' \in \Map(Y,M,\rho', F, \delta, \sigma')$
are $((\rho')_\infty^D,\varepsilon)$-separated. Thus,
\[
N_\varepsilon(\Map(X,M,\rho,F,\delta,\sigma),\rho_\infty^D) = N_\varepsilon(\Map(Y,M,\rho',F,\delta,\sigma'),(\rho')_\infty^D).
\]
Consequently, if $\Sigma = (D_i,\sigma_i)_{i \in I}$ is the strong sofic approximation of $M$, we have
\begin{multline*}
h_\Sigma(X,M) = h_\Sigma(X,M,\rho) = \sup_{\varepsilon > 0}\inf_{F}\inf_{\delta > 0}\limsup_{i \in I} \frac{1}{|D_i|} \log N_\varepsilon(\Map(X,M,\rho,F,\delta,\sigma_i),\rho_\infty^{D_i})\\
= \sup_{\varepsilon > 0}\inf_{F}\inf_{\delta > 0}\limsup_{i \in I} \frac{1}{|D_i|} \log N_\varepsilon(\Map(Y,M,\rho',F,\delta,\sigma_i'),
(\rho')_\infty^{D_i}) = h_\Sigma(Y,M,\rho') = h_\Sigma(Y,M).
\end{multline*}
\end{proof}

Keeping in mind Example \ref{e:rho-X-rho-Y} and the notation therein we have the following monotonicity result.

\begin{proposition}
\label{p:st-ent-subsystem}
Let $X$ be a compact topological space equipped with a continuous action of a sofic monoid $M$.
Let $\Sigma = (D_i,\sigma_i)_{i \in I}$ be a sofic approximation of $M$.
Suppose that $(X,M)$ admits a continuous pseudometric $\rho$.
Let $Y \subset X$ be a closed $M$-invariant subset.
Then one has $h_\Sigma(Y,M,\rho) \leq h_\Sigma(X,M,\rho)$.
\end{proposition}

\begin{proof}
Given a  finite subset $F \subset M$, $\delta > 0$, and a map $\sigma \colon M \to \Map(D)$,
where $D$ is a non-empty finite set,
we have a natural inclusion $\Map(Y,M,\rho,F,\delta,\sigma) \subset \Map(X,M,\rho,F,\delta,\sigma)$ induced by the inclusion
$Y^D \subset X^D$. 
Using~\eqref{e:N-epsilon-for-inv-sub}, it follows that
\[
N_\varepsilon(\Map(Y,M,\rho,F,\delta,\sigma),\rho_\infty^D) \subset  N_\varepsilon(\Map(X,M,\rho,F,\delta,\sigma),\rho_\infty^D)
\]
for every $\varepsilon > 0$.
We deduce that
\begin{align*}
h_\Sigma(Y,M, \rho)
&= \sup_{\varepsilon > 0} \inf_{F} \inf_{\delta > 0} \limsup_{i \in I} \frac{1}{|D_i|} \log N_\varepsilon(\Map(Y,M,\rho,F,\delta,\sigma_i),\rho_\infty^{D_i}) \\
&\leq \sup_{\varepsilon > 0} \inf_{F} \inf_{\delta > 0} \limsup_{i \in I} \frac{1}{|D_i|} \log N_\varepsilon(\Map(X,M,\rho,F,\delta,\sigma_i),\rho_\infty^{D_i}) \\
&= h_\Sigma(X,M, \rho).
\end{align*}
\end{proof} 

\begin{corollary}
\label{c:st-ent-subsystem}
Let $X$ be a compact topological space equipped with a continuous action of an infinite strongly sofic monoid $M$.
Let $\Sigma = (D_i,\sigma_i)_{i \in I}$ be a strong sofic approximation of $M$.
Suppose that $(X,M)$ admits a dynamically generating continuous pseudometric.
Let $Y \subset X$ be a closed $M$-invariant subset.
Then one has $h_\Sigma(Y,M) \leq h_\Sigma(X,M)$.
\end{corollary}

\begin{proof}
If $\rho$ is a dynamically generating continuous pseudometric for $(X,M)$, then we have
$h_\Sigma(Y,M) = h_\Sigma(Y,M,\rho) \leq h_\Sigma(X,M,\rho) = h_\Sigma(X,M)$.
\end{proof}

\section{Sofic topological entropy of shifts and subshifts}
\label{s:ste-ss}
Let $M$ be a sofic monoid, let $A$ be a finite set, and let $X \subset A^M$ be a subshift.
Consider the dynamically generating continuous pseudometric $\rho$ defined by~\eqref{e:pseudo-met-shift}.
Given a non-empty finite set $D$, for all $\varphi, \psi \in X^D$ we have
\begin{equation}
\label{e:rho-2-for-subshifts}
\rho_2^D(\varphi,\psi) = \left(\frac{1}{|D|}\left| \{v \in D : \varphi(v)(1_M) \not=  \psi(v)(1_M) \} \right|\right)^{1/2}
\end{equation}
and
\begin{equation}
\label{e:rho-infty-for-subshifts}
\rho_\infty^D(\varphi,\psi) =
\begin{cases}
0 &\text{ if } \varphi(v)(1_M) = \psi(v)(1_M) \text{ for all } v \in D,\\
1 & \text{ otherwise.}
\end{cases}
\end{equation}

It follows that, given a finite subset $F \subset M$, $\delta> 0$, and $\sigma \colon M \to \Map(D)$,
a map $\varphi \in X^D$ belongs to $\Map(X,M,\rho,F,\delta,\sigma)$ (cf.\ \eqref{e:def-MAP}) 
if and only if for each $m \in F$ the set
\begin{equation}
\label{e:W-F-varphi-m}
W_{\varphi,m} \coloneqq \{v \in D: \left(\varphi(\sigma(m)(v))\right)(1_M) \neq \varphi(v)(m)\}
\end{equation}
satisfies
\begin{equation}
\label{e:W-F-varphi-card-m}
\vert W_{\varphi,m} \vert \leq |D| \delta^2.
\end{equation}
Setting
\begin{equation}
\label{e:W-F-varphi}
W_\varphi \coloneqq \bigcup_{m \in F} W_{\varphi,m} = \{v \in D: \left(\varphi(\sigma(m)(v))\right)(1_M) \neq \varphi(v)(m) \mbox{ for some $m \in F$}\},
\end{equation}
we deduce that if $\varphi \in X^D$ belongs to $\Map(X,M,\rho,F,\delta,\sigma)$ then
\begin{equation}
\label{e:W-F-varphi-card}
\vert W_\varphi \vert < |F|\cdot |D| \delta^2.
\end{equation}

\begin{lemma}
\label{l:entropie-subshift}
Let $M$ be a sofic monoid, let $\Sigma = (D_i,\sigma_i)_{i \in I}$ be a sofic approximation of $M$, and let $A$ be a finite set.
Suppose that $X \subset A^M$ is a subshift.
Then one has
\begin{equation}
\label{e:entropie-subshift}
h_\Sigma(X,M,\rho) \leq \log |A|.
\end{equation}
\end{lemma}
\begin{proof}
It follows from \eqref{e:rho-infty-for-subshifts} that for any $0 < \varepsilon < 1$ and any finite non-empty set $D$, 
a subset $Z \subset X^D$ is $(\rho_\infty^D,\varepsilon)$-separated
if and only if the map $Z \to A^D$ given by $\varphi \mapsto (\varphi(v)(1_M))_{v \in D}$ is injective.
Thus, given a finite subset $F \subset M$, $\delta > 0$, and a map $\sigma \colon M \to \Map(D)$,
taking $Z$ as a maximal $(\rho_\infty^D,\varepsilon)$-separated subset in $\Map(X,M,\rho,F,\delta,\sigma)$ we have
\begin{equation}
\label{e:seconde-inegalite}
N_\varepsilon(\Map(X,M,\rho,F,\delta,\sigma),\rho_\infty^D) = |Z| \leq |A^D| =  |A|^{|D|}.
\end{equation}
Taking logarithms we deduce that for the sofic approximation $\Sigma = (D_i,\sigma_i)_{i \in I}$ of $M$ one has
\begin{align*}
h_\Sigma(X,M,\varepsilon,\rho,F,\delta) 
& = \limsup_{i \in I}\frac{1}{|D_i|} \log N_\varepsilon(\Map(X,M,\rho,F,\delta,\sigma_i),\rho_\infty^D)
\\ & \leq \limsup_{i \in I}\frac{1}{|D_i|} \log |A|^{|D_i|} \\
& = \log |A|,
\end{align*}
and \eqref{e:entropie-subshift} follows.
\end{proof}

The following result extends Proposition~10.28 in~\cite{kerr-li-book}.

\begin{lemma}
\label{l:top-ent-full-shift}
Let $M$ be a sofic monoid, let $\Sigma = (D_i,\sigma_i)_{i \in I}$ be a sofic approximation of $M$, and let $A$ be a finite set.
Then one has 
\begin{equation}
\label{e:entropie-full-subshift}
h_\Sigma(A^M,M,\rho) = \log |A|.
\end{equation}
\end{lemma}
\begin{proof}
Set $X \coloneqq A^M$. Let $F \subset M$ be a finite subset, let $0 < \varepsilon < 1$, let $\delta > 0$,
and let $\sigma \colon M \to \Map(D)$ satisfy Conditions (SM1), (SM2), and (SM3), where $D$ is a non-empty finite set. 

With each $\omega \in A^D$ let us associate $\varphi_\omega \in X^D$ by setting
\[
\varphi_\omega(v)(m) \coloneqq \omega(\sigma(m)(v))
\]
for all $v \in D$ and $m \in M$. 

Note that the map $\omega \mapsto \varphi_\omega$ is injective and that the set $Z \coloneqq \{\varphi_\omega: \omega \in A^D\}$
is $(\rho_\infty^D,\varepsilon)$-separated.
Indeed, if  $\omega_1, \omega_2 \in A^D$ are distinct, then there exists $v \in D$ such that $\omega_1(v) \neq \omega_2(v)$.
As $\sigma(1_M) = \Id_D$ by (SM1), we have
$\varphi_{\omega_1}(v)(1_M) = \omega_1(\sigma(1_M)(v)) =  \omega_1(v) \neq \omega_2(v) = \omega_2(\sigma(1_M)(v)) =
\varphi_{\omega_2}(v)(1_M)$, so that $\rho_\infty^D(\varphi_{\omega_1},\varphi_{\omega_2}) = 1 > \varepsilon$ 
(cf.\ \eqref{e:rho-infty-for-subshifts}).

Moreover, if $v \in D$, we have
\[
\varphi_\omega(\sigma_s(v))(1_M) = \omega(\sigma_{1_M}(\sigma_s(v))) = \omega(\sigma_s(v)) = \varphi_\omega(v)(s) = (s \varphi_\omega(v))(1_M)
\]
so that (cf.\ \eqref{e:rho-2-for-subshifts}),
\[
\rho_2^D(\varphi_\omega \sigma_s, s \varphi_\omega) = 0
\]
for all $\omega \in A^D$ and $s \in F$. 
We deduce that $Z = \{\varphi_\omega: \omega \in A^D\} \subset \Map(X,M,\rho,F,\delta,\sigma)$ and, therefore, keeping in mind that
$Z$ is $(\rho_\infty^D, \varepsilon)$-separated,
\begin{equation}
\label{e:premiere-inegalite}
|A|^{|D|} = |A^D| = |\{\varphi_\omega: \omega \in A^D\}| \leq N_\varepsilon(\Map(X,M,\rho,F,\delta,\sigma),\rho_\infty^D).
\end{equation}
Taking logarithms we deduce that for the sofic approximation $\Sigma = (D_i,\sigma_i)_{i \in I}$ of $M$ one has
\begin{align*}
h_\Sigma(X,M,\varepsilon,\rho,F,\delta) & = \limsup_{i \in I}\frac{1}{|D_i|} \log N_\varepsilon(\Map(X,M,\rho,F,\delta,\sigma_i),\rho_\infty^D)
\\ & \geq \limsup_{i \in I}\frac{1}{|D_i|} \log |A|^{|D_i|} \\
& = \log |A|
\end{align*}
and, consequently, $h_\Sigma(X,M,\rho) \geq \log |A|$.
\par
On the other hand, by virtue of Lemma~\ref{l:entropie-subshift}, we have $h_\Sigma(X,M, \rho) \leq \log |A|$, and \eqref{e:entropie-full-subshift} follows.
\end{proof}

The following result extends Proposition 10.29 in \cite{kerr-li-book}.

\begin{theorem} 
\label{t:monotonicity}
Let $M$ be an infinite strongly sofic monoid and let $\Sigma=(D_i, \sigma_i)_{i \in I}$ be a strong sofic approximation of $M$.  
Let $A$ be a finite set and let $X \subset A^M$ be a subshift.
Then one has $h_\Sigma(X,M) \leq \log |A|$. 
Moreover, one has $h_\Sigma(X,M) = \log |A|$ if and only if $X = A^M$.
\end{theorem}
\begin{proof}
Since $h_\Sigma(X,M) \leq \log |A|$ by Lemma~\ref{l:entropie-subshift} and $h_\Sigma(A^M,M) = \log |A|$ by Lemma~\ref{l:top-ent-full-shift}, we only need to prove that if $X \subsetneqq A^M$ then $h_\Sigma(X,M) < \log |A|$.
Thus, suppose that $X \subsetneqq A^M$. Then there exists a finite subset $F \subset M$
and an $X$-forbidden pattern $p \in A^F \setminus X_F$. After enlarging $F$ if necessary, we may suppose that $|F| \geq 2$.
\par
Let $0 < \delta \leq 1/(2(|F|+1))$ and let $0 < \varepsilon \leq \delta/\binom{|F|}{2}$ (cf.\ Remark~\ref{r:entropie-epsilon-limite}).
Since $M$ is strongly sofic, we can find an integer $\Delta_F \geq 1$ (not depending on $\varepsilon$), a non-empty finite set $D$, and a map $\sigma \colon M \to \Map(D)$ satisfying Conditions (SM1)-(SM4) relative to the finite subset $K \coloneqq F$ and $\varepsilon$. 
\par
Choose a $(\rho_\infty^D,\varepsilon)$-separated subset $Z \subset \Map(X,M,\rho,F,\delta,\sigma)$
of maximal cardinality, so that $|Z| = N_\varepsilon(\Map(X,M,\rho,F,\delta,\sigma),\rho_\infty^D)$.
With every $\varphi \in Z$, let us associate the element $\omega_\varphi \in A^D$ defined by 
$\omega_\varphi(v) \coloneqq \varphi(v)(1_M)$ for all $v \in D$.
If $\varphi, \psi \in Z$ are distinct, then, as $Z$ is $(\rho_\infty^D, \varepsilon)$-separated, 
there exists $v \in D$ such that $\varphi(v)(1_M) \neq \psi(v)(1_M)$, that is, $\omega_\varphi(v) \neq \omega_\psi(v)$. 
This shows that the map $Z \to A^D$, given by $\varphi \mapsto \omega_\varphi$, is injective.
\par
Let $V \subset D$ denote the set of all $v \in D$ such that the map $F \to D$, given by $s \mapsto \sigma(s)(v)$, is not injective.
We deduce from (SM3) that
\begin{align*}
|V| & \leq \sum_{\substack{s,t \in F\\s \neq t}} |\{v \in D : \sigma(s)(v)  = \sigma(t)(v)\}|  \\
& = \sum_{\substack{s,t \in F\\s \neq t}} \left(1 -  d^\Ham_D (\sigma(s), \sigma(t)) \right) |D|  \\
& \leq \binom{|F|}{2} \varepsilon |D| \leq \delta |D|.
\end{align*}
Consequently, we have
\begin{equation}
\label{e:ubound-for-V}
|V| \leq \delta |D|.
\end{equation}
\par
Consider now the set $\WW$ consisting of all subsets $W \subset D$ such that $|W| \leq (|F|+1)\delta |D|$ and $V \subset W$. 
Note that from our choice of $\delta$ we deduce that for all $W \in \WW$, we have 
\begin{equation}
\label{e:bound-cardinal-de-D-W}
|D \setminus W| = |D| - |W| \geq |D| - (|F|+1)\delta |D| \geq  \frac{|D|}{2}. 
\end{equation}

Fix $W \in \WW$.  
Take a subset $U_W \subset D \setminus W$ of maximal cardinality such that $\sigma(F)(u) \cap \sigma(F)(v) = \varnothing$ for all distinct
$u,v \in U_W$. 
We claim that $\left(\sigma(F)^{-1}(\sigma(F)(u))\right)_{u \in U_W}$ covers $D \setminus W$ (here we use the notation introduced in 
\eqref{e:inverse-set-map}). Indeed, let $v \in D \setminus W$. Suppose that
$v \notin \bigcup_{u \in U_W} \sigma(F)^{-1}(\sigma(F)(u))$, that is, $v \notin \sigma(f_1)^{-1}(\sigma(f_2)(u))$ for all $f_1, f_2 \in F$
and $u \in U_W$. Then, for all $u \in U_W$ we have $\sigma(f_1)(v) \neq \sigma(f_2)(u)$ for all $f_1, f_2 \in F$, that is, 
$\sigma(F)(v) \cap \sigma(F)(u) = \varnothing$, contradicting the maximality of $U_W$. The claim follows.
\par
Combining with the inequalities $|\sigma(s)^{-1}(v)|\leq \Delta_F$ for all $s \in F$ and $v \in D$ (cf.\ Condition (SM4)), we find that 
\begin{align*}
|D\setminus W| \leq \sum_{u \in U_W} |\sigma(F)^{-1}(\sigma(F)(u))| \leq \sum_{u \in U_W}  \sum_{s \in F}  \sum_{t \in F}  
|\sigma(s)^{-1}(\sigma(t)(u))|\leq |U_W| \cdot|F|^2 \Delta_F. 
\end{align*}
Keeping in mind \eqref{e:bound-cardinal-de-D-W}, we obtain 
\begin{equation}
\label{e:cardinal-U}
|U_W| \geq \frac{|D\setminus W|}{|F|^2\Delta_F} \geq \frac{|D|}{2|F|^2 \Delta_F}.
\end{equation}

Define $Z_W$ to be the set of all $\varphi \in Z$ such that
\begin{equation}
\label{e:def-z-varphi}
\left(\varphi(\sigma(s)(v))\right)(1_M) = \left(s\varphi(v)\right)(1_M) \ \ \mbox{ for all $v \in D \setminus W$ and all $s \in F$}.
\end{equation}
Thus, if $\varphi \in Z_W$ and $u \in U_W$, the pattern $q = q_{\varphi, u} \in A^F$ defined by 
\[
q(s) \coloneqq \omega_\varphi(\sigma(s)(u))
\]
satisfies
\[
q(s) = \omega_\varphi(\sigma(s)(u)) = \varphi(\sigma(s)(u))(1_M) = (s\varphi(u))(1_M) = \varphi(u)(s)
\]
for all $s \in F$ (cf.\ \eqref{e:def-z-varphi}), showing that 
$q = \varphi(u)\vert_F \in X_F$. Therefore $q \neq p$ (the forbidden $X$-pattern in the beginning of the proof).

As a consequence, given $u \in U_W$ we have 
\begin{equation}
\label{e:new}
|\{\omega_\varphi\vert_{\sigma(F)(u)}: \varphi \in Z_W\}| = |\{q_{\varphi, u}: \varphi \in Z_W\}| \leq |A^F \setminus \{p\}| = |A^F|-1 = 
|A|^{|F|}-1.
\end{equation} 
Since $U_W \subset D \setminus W \subset D \setminus V$, the map $F \to D$, given by $s \mapsto \sigma(s)(u)$, 
is injective for all $u \in U_W$.
Consequently, we have $|\sigma(F)(u)| = |F|$ for all $u \in U_W$. 
As the subsets $\sigma(F)(u)$, $u \in U_W$, are mutually disjoint, it follows that
\begin{equation}
\label{e:card-sigmaFUW}
|\sigma(F)(U_W)| = |F|\cdot |U_W|.
\end{equation}
We deduce that
\[
|\{\omega_\varphi\vert_{\sigma(F)(U_W)}: \varphi \in Z_W\}| \leq (|A|^{|F|}-1)^{|U_W|}.
\]
From the injectivity of the map $Z_W \to A^D = A^{D \setminus \sigma(F)(U_W)} \times A^{\sigma(F)(U_W)}$ given by 
\[
\varphi \mapsto \omega_\varphi = (\omega_\varphi\vert_{A^{D \setminus \sigma(F)(U_W)}}, \omega_\varphi\vert_{A^{\sigma(F)(U_W)}}),
\] 
and using \eqref{e:card-sigmaFUW}, we then deduce the estimate
\[
|Z_W| \leq |A|^{|D| - |F|\cdot|U_W|} (|A|^{|F|}-1)^{|U_W|} = |A|^{|D|} (1-|A|^{-|F|})^{|U_W|}.
\]
By virtue of \eqref{e:cardinal-U}, we have that 
\begin{equation}
\label{e:upper-bound-Z-W}
|Z_W| \leq |A|^{(1 - \beta_0)|D|},
\end{equation}
where $\beta_0 \displaystyle \coloneqq -\frac{\log_{|A|}(1-|A|^{-|F|})}{2\Delta_F |F|^2} > 0$. 
Note that $\beta_0$ is independent of $W, \delta, D$, and $\sigma$.

Let $\varphi \in Z$. As $\delta < 1$, from \eqref{e:W-F-varphi-card} we deduce that $|W_\varphi| \leq |F|\delta |D|$.
As $|V| \leq \delta |D|$ by~\eqref{e:ubound-for-V}, we have that $W(\varphi) \coloneqq W_\varphi \cup V$ satisfies
$|W(\varphi)| \leq (|F|+1)\delta|D|$, so that $W(\varphi) \in \WW$.
Let $v \in D \setminus W(\varphi) \subset D \setminus W_\varphi$. From \eqref{e:W-F-varphi} it follows that
$\varphi(\sigma(m)(v))(1_M) = \varphi(v)(m)$ for all $m \in F$. From \eqref{e:def-z-varphi} we deduce
that $\varphi \in Z_{W(\varphi)}$. It follows that
\begin{equation}
\label{e:Z-union-ZW}
Z = \bigcup_{W \in \WW} Z_W.
\end{equation}
\par
Setting $d \coloneqq |D|$ and $t \coloneqq (|F|+1)\delta$, we have $0 < t <1/2$ so that, by applying Lemma~\ref{lemma:Stirling},
there exist $d_0 = d_0(t) > 0$ and $\beta = \beta(t) > 0$, with $\beta(t) \to 0$ as $\delta \to 0$,
such that the number of all subsets $W \subset D$ with cardinality  $|W| \leq \floor{(|F|+1)\delta |D|} = \floor{t d}$, is bounded above, provided $D$ is large enough (namely, $|D| = d \geq d_0$), by $e^{\beta |D|}$.
\par
From \eqref{e:Z-union-ZW} and \eqref{e:upper-bound-Z-W} we then deduce the estimate
\[
|Z| = \vert \bigcup_{W \in \WW} Z_W \vert \leq e^{\beta |D|} |A|^{(1 - \beta_0)|D|}.
\]
Recalling that $Z$ is a maximal $(\rho_\infty^D,\varepsilon)$-separated subset of $\Map(X,M,\rho,F,\delta,\sigma)$, we then have
\[ 
\frac{1}{|D|}\log N_\varepsilon(\Map(X,M,\rho,F,\delta,\sigma), \rho_\infty^D) = \frac{1}{|D|}\log |Z| \leq \beta + (1 - \beta_0) \log |A|.
\]
Taking the infimum over $\delta$ (equivalently, $\lim_{\delta \to 0}$, by Remark~\ref{r:entropie-epsilon-limite}), over $F$, and then the supremum over $\varepsilon$ (equivalently, $\lim_{\varepsilon \to 0}$, by Remark~\ref{r:entropie-epsilon-limite}), and keeping in mind that 
$\beta(t) \to 0$ as  $t \to 0$, we see that $h_\Sigma(X,M) = h_\Sigma(X,M,\rho) \leq (1 - \beta_0) \log |A| < \log|A|$.
\end{proof}

We are now in a position to prove the main result of our paper.

\begin{proof}[Proof of Theorem \ref{t:main}]
Let $M$ be a strongly sofic monoid. Let $A$ be a finite set and let $\tau \colon A^M \to A^M$ be an injective cellular automaton.
We have to show that $\tau$ is surjective.
\par
If $M$ is a finite monoid then $A^M$ is a finite set and the result is trivial.
Thus, we may assume that $M$ is infinite.
Set $X \coloneqq \tau(A^M)$. Then $X$ is a subshift of $A^M$ and $\tau$ yields a topological conjugacy between the dynamical systems
$(A^M,M)$ and $(X,M)$. Consequently, if $\Sigma$ is a strong sofic approximation of $M$, we have $h_\Sigma(A^M,M) = h_\Sigma(X,M)$ 
by Theorem~\ref{t:topo-invariant}. It then follows from Theorem~\ref{t:monotonicity} that $X = A^M$.
Therefore $\tau$ is surjective.
\end{proof}

The strongly sofic monoid in Example~\ref{ex:no-group} is right-cancellative but not left-cancellative.
By classical results of Malcev \cite{malcev1, malcev2} (see also \cite{edwardes-heath}), 
there exist cancellative monoids which are not embeddable in a group.
This motivates the following question, for which we have not been able to provide an answer.

\begin{question}
Does there exist a cancellative strongly sofic monoid which is not embeddable in a group?
\end{question}


\end{document}